\documentclass[11pt]{amsart}

\usepackage{amsfonts}
\usepackage{amscd}
\usepackage{amssymb}
\allowdisplaybreaks

\usepackage{color}
\usepackage{url}

\newtheorem{thm}{Theorem}[section]
\newtheorem{cor}[thm]{Corollary}
\newtheorem{lem}[thm]{Lemma}

\newtheorem{prop}[thm]{Proposition}
\theoremstyle{definition}

\theoremstyle{remark}
\newtheorem{rem}[thm]{Remark}
\numberwithin{equation}{section}

\newcommand{\R}{\mathbb R}

\newcommand{\Na}{\mathbb N}
\newcommand{\He}{\mathbb H}

\newcommand{\si}{\sigma}

\newcommand{\al}{\alpha}

\newcommand{\la}{\lambda}
\newcommand{\C}{{\mathbb C}}

\newcommand{\pa}{\partial }
\newcommand{\La}{\langle}
\newcommand{\Ra}{\rangle}

\newcommand{\Hc}{\mathcal H}

\newcommand{\sn}{{\mathbb S}^{n-1}}
\newcommand{\sd}{{\mathbb S}^{d-1}}

\newcommand{\om}{ \omega}
\newcommand{\D}{\Delta}
\newcommand{\Dk}{\Delta_{\kappa}}
\newcommand{\dl}{\delta }

\newcommand{\N}{\nabla }
\newcommand{\K}{\kappa }
\newcommand{\gm}{\gamma}
\newcommand{\Gm}{\Gamma}

\renewcommand{\Re}{\operatorname{Re}}

\title[Hardy inequalities for Dunkl--Hermite operators]{Hardy-type inequalities for fractional powers of the\\ Dunkl--Hermite operator}

\author[\'O. Ciaurri, L. Roncal, and S. Thangavelu]{\'Oscar Ciaurri, Luz Roncal and Sundaram Thangavelu}

\address[\'O. Ciaurri]{Departamento de Matem\'aticas y Computaci\'on\\
    Universidad de La Rioja\\
    26006 Logro\~no, Spain}
\email{oscar.ciaurri@unirioja.es}

\address[L. Roncal]{Departamento de Matem\'aticas y Computaci\'on\\
    Universidad de La Rioja\\
    26006 Logro\~no, Spain and BCAM - Basque Center for Applied Mathematics\\
 48009 Bilbao, Spain}
\email{lroncal@bcamath.org}

\address[S. Thangavelu]{Department of Mathematics\\
 Indian Institute of Science\\
560 012 Bangalore, India}

\email{veluma@math.iisc.ernet.in}

\thanks{The first and second authors were
supported by grant MTM2015-65888-C04-4-P from Spanish Government. The second author is supported by the Basque Government through the BERC 2014-2017 program and by Spanish Ministry of Economy and Competitiveness MINECO: BCAM Severo Ochoa excellence accreditation SEV-2013-0323. The third author is supported by J.C. Bose Fellowship from D.S.T., Government of India}

\begin{document}

\maketitle

\begin{abstract}
We prove Hardy-type inequalities for a fractional Dunkl--Hermite operator which incidentally give  Hardy inequalities for the fractional harmonic oscillator as well. The idea is to use  $h$-harmonic expansions  to reduce the problem in  the Dunkl--Hermite context to the Laguerre setting. Then, we push forward a technique based on a non-local ground representation, initially developed by R. L. Frank, E. H. Lieb and R. Seiringer in the Euclidean setting, to get a Hardy inequality for the fractional-type Laguerre operator. The above-mentioned method is shown to be adaptable to an abstract setting, whenever there is a ``good'' spectral theorem and an integral representation for the fractional operators involved.
\end{abstract}
\section{Introduction and main results}

The original Hardy's inequality for the Laplacian on $ \R^d $ says that
$$
\frac{(d-2)^2}{4} \int_{\R^d} \frac{|f(x)|^2}{|x|^{2}} \, dx \leq  \int_{\R^d} |\N f(x)|^2 \, dx, \qquad d\ge3,
$$
where $\nabla$ stands for the gradient. The statement of Hardy's inequality can be generalized to other contexts and operators. Let $(X,d\eta)$ be a measure space, where $d\eta$ is a positive measure on $X$. Given $0<\si<1$, let us denote by $L^{\sigma}$ the fractional powers of a non-negative, self-adjoint operator $L$ on $L^2(X)$. We are interested in Hardy-type inequalities of the form
\begin{equation*}
\int_{X} \frac{|f(x)|^2}{(1+|x|^2)^\si} \, d\eta(x) \leq B_{\si}  \langle L^{\si}f, f\rangle,
\end{equation*}
for certain constant $B_{\si}$, or Hardy-type inequalities where the potential involved is homogeneous. Namely,
for $0<\si<1$, $f \in C_0^\infty(X) $ we are also concerned with inequalities of the type
\begin{equation}
\label{eq:HardyHomo}
\int_{X} \frac{|f(x)|^2}{|x|^{2\si}} \, d\eta(x) \leq C_{d,\si}   \langle L^{\si}f, f\rangle.
\end{equation}
When $(X,d\eta)=(\R^d,dx)$ and $L=-\Delta$, the sharp constant $ C_{d,\si} $ in \eqref{eq:HardyHomo} is already known \cite{B4, Herbst, Yafaev}. Nevertheless, R. L. Frank, E. H. Lieb, and R. Seiringer \cite{FLS} found a different proof of the inequality \eqref{eq:HardyHomo} in the Euclidean setting when $ 0 < \si < \min\{1, d/2\} $ by using a non-local version of the \textit{ground state representation}. Such representation improved the previous results in the sense that it provided quantitative information on the error by delivering a remainder term, see e.g. Theorem~\ref{thm:groundstate} in the general setting. Hardy inequalities and their generalizations are of relevant importance in mathematical analysis. They may be applied, among others, to problems concerning mathematical physics, partial differential equations, spectral theory, harmonic analysis, and potential theory.

The aim of this paper is two-fold. First, we revisit the technique developed by Frank--Lieb--Seiringer \cite{FLS} (see also \cite{FS}). We present an exposition of  such procedure in a systematic and unified way so that the method can be applied in more general settings. This description will be developed in Section \ref{sec:generalMethod}. Secondly, we will apply this method to get Hardy inequalities in several settings. In particular, we will prove Hardy-type inequalities for certain fractional Laguerre operators $L_{\al,\si}$, and from these we will deduce the Hardy inequalities for the corresponding fractional Dunkl--Hermite operator $\mathbf{H}_{\K,\si}$. The definitions of these operators can be found, respectively, in \eqref{eq:spectralLsi} and Subsection~\ref{subse:HardyHK} below. As an immediate consequence, we will also obtain Hardy inequalities for the fractional harmonic oscillator. Observe that in \cite{RT} the adapted method was already applied to obtain Hardy inequalities on the Heisenberg group.

We notice that a sharp Pitt's inequality for the fractional powers of the Dunkl operator in $L^2(\R^d)$ was recently proved by D. V. Gorbachev, V. I. Ivanov, and S. Yu. Tikhonov in \cite{GIT}. Such Pitt's inequality, which is a weighted norm inequality for the Dunkl transform on $ \R^d ,$ can be rewritten as a Hardy type inequality for fractional powers of the Dunkl-Laplacian $ \Delta_{\K}.$ Thus the inequality takes the form
$$ \int_{\R^d} \frac{|f(x)|^2}{|x|^{2\si}} \, h_\K(x) d(x) \leq B_{\si}  \langle (-\Delta_\K)^{\si}f, f\rangle$$
with a sharp constant $ B_\si.$
Our guess is that an expression for the error in such an inequality could be accomplished by using the ground state representation.

We will not work with the pure fractional powers of the operators under consideration in this paper, but with their \textit{conformally invariant} fractional powers, which will suit better for our purposes. As a consequence of the Hardy inequalities obtained for these operators, we will be able to deduce Hardy inequalities for the corresponding pure fractional powers. We borrow the terminology ``conformally invariant fractional powers'' from the context of  sublaplacian on Heisenberg groups, see \cite{RT}. As this group arises as the boundary of Siegel's upper half space, conformally invariant operators make perfect sense, see \cite{CG, GZ} for the exact definition of this invariance. The spectral theory of sublaplacian $ \mathcal{L} $ on the Heisenberg grop $ \He^d $  is closely related to Laguerre operators  of type $ (d-1) $. More precisely, the action of $ \mathcal{L} $ on functions of the form $ e^{it}f(|z|) $ is given by the action of the Laguerre operator $ L_{d-1}$ on $ f.$  Consequently,  when we consider the conformally invariant fractional power $ \mathcal{L}_\sigma $ acting on such functions, we obtain certain fractional powers of the Laguerre operator which we call conformally invariant.  The analogues of such operators for any $ \alpha > -\frac{1}{2}$  will be defined via spectral theorem and they will have a superficial resemblance with conformally invariant fractional powers of the sublaplacian on the Heisenberg group.

As it will be explained in Section \ref{sec:generalMethod}, some of the main features of the general way to proceed consist of having at hand both a ``good'' spectral theorem and an integral representation for the ground state representation. In the case of Dunkl--Hermite setting we do not have a convolution structure (neither we have an explicit formula for the fundamental solution, which seems to be playing an important role). So our idea is to use $h$-harmonics to reduce the Dunkl--Hermite setting to the Laguerre case, where we have a convolution structure, viz., the Laguerre convolution is at our disposal.  Moreover, in the latter setting we have ``good'' spectral theorem (in our case this means to have analogous results to those ones proved by M. Cowling and U. Haagerup in \cite[Section 3]{CH}, and slightly generalized by the second and third authors in \cite[Section 3]{RT}) and explicit fundamental solutions. This is why we will focus on the Laguerre differential operator and as an application we will deduce the results for the Dunkl--Hermite operator (which includes the ordinary Hermite operator as a particular case).

Now we state the Hardy inequalities for our operators. For $ \alpha > -1/2,$ the Laguerre differential operator defined by
$$
L_{\al}=-\frac{d^2}{dr^2}+r^2-\frac{2\al+1}{r}\frac{d}{dr}
$$
 is symmetric on $ L^2(0,\infty)$ taken with respect to the measure $d\mu_\al(r) = r^{2\al+1} dr$. For $ 0 < \sigma < 1 $ we use spectral theorem to define the conformally invariant fractional powers $ L_{\al,\sigma}$, see Section \ref{sec:HardyLaguerre} for the precise definition. For  $\dl>0$ and $\al>-1/2$, we let
\begin{equation}
\label{eq:weight}
w_{\al,\si}^{\dl}(r):=c_{\al,\si}(\dl+r^2)^{-(\al+1+\si)/2}K_{(\al+1+\si)/2}\big((\dl+r^2)/2\big),
\end{equation}
where $K_{\nu}$ is the Macdonald's function of order $\nu$ (see (\cite[Chapter 5, Section 5.7]{Lebedev} for the definition of such function), and $c_{\al,\si}$ to be the constant
\begin{equation}
\label{eq:calsi}
c_{\al,\si}:=\frac{\sqrt{\pi}2^{1-\si}}
{\Gamma[(\al+2+\si)/2]}.
\end{equation}
The first Hardy inequality concerns the fractional Laguerre operators $ L_{\al,\sigma}$. Let us define the constant $B_{\al,\si}^{\dl}$ by
\begin{equation}
\label{eq:Bconstant}
B_{\al,\si}^{\dl}:= \dl^{\si}\frac{\Gamma\big(\frac{\al+2+\si}{2}\big)}
{\Gamma\big(\frac{\al+2-\si}{2}\big)}.
\end{equation}
\begin{thm}[Hardy inequality for the Laguerre operator]
\label{thm:HardyLaguerre}
Let $0<\si<1$, $\dl>0$, and $\al>-1/2$. Then
$$
B_{\al,\si}^{\dl}\int_{0}^{\infty}\frac{|f(r)|^2}{(\dl+r^2)^{\si}}\,d\mu_{\al}(r)\le \frac{4^{\si}}{\dl^{\si}}(B^{\dl}_{\al,\si})^2 \int_{0}^{\infty}|f(r)|^2\frac{w_{\al,\si}^{\dl}(r)}{w_{\al,-\si}^{\dl}(r)}\,d\mu_{\al}(r)\le \langle L_{\al,\si} f,f\rangle_{d\mu_{\al}}
$$
for all functions $f\in C_0^{\infty}(0,\infty)$.
\end{thm}
The constant just after the first inequality above turns out to be sharp. Actually, the function $w_{\al,\si}^{\dl}$ is an optimizer, in view of Theorem \ref{prop:Lsweight}.

We use the results of the above theorem in order to prove the following Hardy inequality for fractional powers of the Dunkl-Hermite operators $ \mathbf{H}_{\K} = -\Delta_{\K} +|x|^2 $ on $ \R^d.$ Here $ \K $ is the multiplicity function defined on the  Coxeter group associated to a given root system and $ \Delta_{\K} $ is the Dunkl Laplacian on $ \R^d.$  We refer to Section \ref{sec:hardyDH} for the precise definition of these operators and their (conformally invariant) fractional powers $ \mathbf{H}_{\K,\si}.$
For a given Coxeter group $ W $  and a non-negative multiplicity function $ \K $ let $ \gm = \sum_{w \in W}  \K(w)$  and define the following constant
\begin{equation}
\label{eq:lambda0}
\lambda:=\frac{d}{2}+\gm-1.
\end{equation}
Details  about these notations and definitions for the theorem below can be found in Section~\ref{sec:hardyDH}.
\begin{thm}[Hardy inequality for the Dunkl--Hermite operator]
\label{thm:HardyDunkl-Hermite}
Let $0<\si<1$, $\dl>0$, and $\la$ be as in \eqref{eq:lambda0}. Then
$$
B_{\la,\si}^{\dl} \int_{\R^d}\frac{|f(x)|^2}{(\dl+|x|^2)^{\si}}\,h_{\K}(x)\,dx\le \langle \mathbf{H}_{\K,\si} f,f\rangle_{L^2(\R^d,h_{\K}) }
$$
for all functions $f\in C_0^{\infty}(\R^d)$, where $B_{\la,\si}^{\dl}$ is as in \eqref{eq:Bconstant}.
\end{thm}
When $ \K = 0 $ the Dunkl-Hermite operator $ \mathbf{H}_{\K} $ reduces to the standard Hermite operator and hence  from Theorem \ref{thm:HardyDunkl-Hermite}, we immediately obtain the following result for  $H_{\si},$  the conformally invariant fractional power of the Hermite operator (or harmonic oscillator) $ H$.
\begin{cor}[Hardy inequality for the harmonic oscillator]
\label{cor:HardyHO}
Let $0<\si<1$ and $\dl>0$. Then
$$
B_{(d/2-1), \si}^{\dl}\int_{\R^d}\frac{|f(x)|^2}{(\dl+|x|^2)^{\si}}\,dx\le \langle H_\si f,f\rangle_{L^2(\R^d) }
$$
for all functions $f\in C_0^{\infty}(\R^d)$, where $B_{(d/2-1), \si}^{\dl}$ is as in \eqref{eq:Bconstant}.
\end{cor}
From the results just stated, we can deduce Hardy inequalities for pure fractional powers of the underlying operators. Define
\begin{equation}
\label{eq:Salpha}
S^{\al,\si}_n:=\frac{\Gamma\big(\frac{4n+2\al+2}{4}
+\frac{1+\si}{2}
\big)}
{\Gamma\big(\frac{4n+2\al+2}{4}+\frac{1-\si}{2}\big)}.
\end{equation}
With the same notation as in the previous theorems and corollary, let $U_{\si}$ be the operator defined as $U_{\si}:=L_{\al,\si}L_{\al}^{-\si}$. It can be shown that this operator is bounded, \and the operator norm $\|U_{\si}\|$ can be expressed explicitly.  Indeed, it is easy to see that
$$
\|U_\si\| = \sup_{n\geq 0} \big( (4n+2\al+2)^{-\si} 4^{\si}S_n^{\al,\si} \big).
$$
Further, by using Stirling's formula for the Gamma function, one can check that $\|U_{\si}\|\sim 1$, for $\alpha$ large enough.
Then we have a Hardy inequality for $L_{\al}^{\si}$ in the following corollary.
\begin{cor}
\label{cor:HardyLaguerre}
Let $0<\si<1$, $\dl>0$, and $\al>-1/2$. Then
$$
B_{\al,\si}^{\dl}\int_{0}^{\infty}\frac{|f(r)|^2}{(\dl+r^2)^{\si}}\,d\mu_{\al}(r)\le \frac{4^{\si}}{\dl^{\si}}(B^{\dl}_{\al,\si})^2 \int_{0}^{\infty}|f(r)|^2\frac{w_{\al,\si}^{\dl}(r)}{w_{\al,-\si}^{\dl}(r)}\,d\mu_{\al}(r)\le \|U_{\si}\| \langle L_{\al}^{\si} f,f\rangle_{d\mu_{\al}}
$$
for all functions $f\in C_0^{\infty}(0,\infty)$, where $B_{\al, \si}^{\dl}$ is as in \eqref{eq:Bconstant}.
\end{cor}
Analogously, if we consider the operator $V_{\K,\si}:=\mathbf{H}_{\K,\si} \mathbf{H}_{\K}^{-\si} $, we get the corresponding result for the fractional powers of the Hermite--Dunkl operator
\begin{cor}
\label{cor:HardyDunkl-Hermite}
Let $0<\si<1$, $\dl>0$, and $\la$ be as in \eqref{eq:lambda0}. Then
$$
B_{\la,\si}^{\dl} \int_{\R^d}\frac{|f(x)|^2}{(\dl+|x|^2)^{\si}}\,h_{\K}(x)\,dx\le \|V_{\K,\si}\|\langle \mathbf{H}_{\K}^{\si} f,f\rangle_{L^2(\R^d,h_{\K}) }
$$
for all functions $f\in C_0^{\infty}(\R^d)$, where $B_{\la, \si}^{\dl}$ is as in \eqref{eq:Bconstant}.
\end{cor}
From the above, we also obtain the result for the fractional powers of the harmonic oscillator.
\begin{cor}
\label{corcor:HardyHO}
Let $0<\si<1$ and $\dl>0$. Then
$$
B_{(d/2-1), \si}^{\dl}\int_{\R^d}\frac{|f(x)|^2}{(\dl+|x|^2)^{\si}}\,dx\le \|V_{0,\si}\|\langle H^\si f,f\rangle_{L^2(\R^d,dx) }
$$
for all functions $f\in C_0^{\infty}(\R^d)$, where $B_{(d/2-1), \si}^{\dl}$ is as in \eqref{eq:Bconstant}.
\end{cor}

\begin{rem}
Observe that all the results above are stated for smooth functions with compact support in the corresponding spaces. However, they could be extended to suitable Sobolev spaces, as in \cite{FLS} or \cite{RT}. Since the procedure is quite standard we leave details to the interested reader.
\end{rem}

\begin{rem}
Uncertainty principles could also be deduced from the Hardy inequalities as in \cite[Corollary 1.7]{RT}.
\end{rem}

\begin{rem} In \cite{RT} the second and third authors have also obtained Hardy-type inequalities where the weight functions are homogeneous. We can therefore follow the same method in obtaining such inequalities for Laguerre and Dunkl--Hermite operators. However, as the resulting inequalities are not known to be sharp even in the case of the Heisenberg group we do not pursue them here. Definitively, a quick inspection shows that a Hardy-type inequality with homogeneous weight cannot be obtained simply by letting $\delta$ go to zero in the Hardy inequalities stated above. On the other hand, the Laguerre setting can be viewed as arising from the Heisenberg setting when the functions considered are radial in the first variable. In this way, the weight $w_{\alpha,\sigma}^{\delta}$ appears in connection with certain function $u_{\si,\dl}$ (see Subsection \ref{subsec:CH}) which is closely related to the function $u_{\lambda,\dl}(v,z)$ in \cite[p.530]{CH}. When $\delta$ goes to zero, $u_{\lambda,\dl}$ tends (in a distributional sense) to the kernel of the intertwining operator of A. W. Knapp and E. M. Stein \cite{KS}. In its turn, the non-homogeneous function $u_{\lambda,1}$ is connected with a Poisson-type kernel when considering groups of type $H$ which include the nilpotent components of the Iwasawa decomposition, see \cite[Section 3]{CH} (also \cite[Ch. IX Theorem 3.8]{Hel}).
\end{rem}

The outline of the paper is as follows. In Section \ref{sec:generalMethod} we describe  the general method to get a Hardy-type  inequality for  fractional powers of  operators in an abstract setting. Such a procedure is then applied to  Laguerre operators (see the proof of Theorem \ref{thm:HardyLaguerre} at the end of Section \ref{sec:HardyLaguerre}). We will then use the Hardy inequality in the Laguerre context to obtain a Hardy inequality for the corresponding fractional-type Dunkl--Hermite operator in Section \ref{sec:hardyDH}. As explained earlier, the idea is to use expansions in $h$-harmonics. Moreover, we will prove a Hecke--Bochner identity that allows us to write the Dunkl--Hermite projections in terms of Laguerre convolutions which plays a crucial role in the proof. Such identity is of independent interest. Once the reduction is achieved,  Theorem~\ref{thm:HardyDunkl-Hermite} will be proved at the end of Section \ref{sec:hardyDH}.

\section{The general method}
\label{sec:generalMethod}

In this section we consider a measure space $(X,d\eta)$, where $d\eta$ is a positive measure on a smooth Riemannian manifold $X$. The function spaces $L^p(X)$ are understoood to be taken with respect to $ d\eta.$ Let $A$ be a non-negative, self-adjoint operator on $L^2(X)$. Then, there is a unique resolution $E$ of the identity, supported on the spectrum of $A$ (which is a subset of $[0,\infty)$), so that the spectral resolution of $A$ is given by
$$
A = \int_0^\infty \lambda dE(\lambda),
$$
or, equivalently,
$$
\langle Af,g\rangle = \int_0^{\infty} \lambda  dE_{f,g}(\lambda), \quad f\in \operatorname{Dom}A, \,\,\, g\in L^2(X).
$$
Here, $dE_{f,g}(\lambda)$ is a regular  complex Borel measure of bounded variation concentrated on the spectrum of $A$, and we use the notation $\langle f,g\rangle=\int_Xf(x)\overline{g(x)}\,d\eta(x)$. The role of Plancherel formula will be played by
$$
\langle Af, Af\rangle = \int_0^\infty \lambda^2  dE_{f,f}(\lambda ).
$$

Let $0<\si<1$. In the most general formulation of the method, we will denote by $\Lambda_{\si}$ a fractional-type operator related to $A$. Our aim is to prove a Hardy´-type inequality of the form
\begin{equation}
\label{eq:hardy}
\langle\Lambda_\si f,f\rangle \ge C_{\si}\int_X h_\si(x)|f(x)|^2\,d\eta(x),
\end{equation}
where $h_{\si}(x)$ is an appropriate positive function, and we look for the explicit positive constant $C_{\si}$ to be sharp.

There are two steps to be furnished in the method we present to obtain the Hardy-type inequality:

\begin{enumerate}
\item To get and integral representation for $\Lambda_\si f$ and an explicit form for $\langle\Lambda_\si f,f\rangle$.
\item To write the ground state representation and use the expression for  $\langle\Lambda_\si f,f\rangle$ obtained previously.
\end{enumerate}

\subsection{Integral representation and the expression for  $\langle\Lambda_{\si}f,f\rangle$}

Let us assume that we have an integral representation of the following form for the operator $\Lambda_\si$, which is valid for all $f\in C_0^{\infty}(X)$:
\begin{equation}
\label{eq:integral-rep}
\Lambda_\si f(x)=a_{\si}\int_X\big(f(x)-f(y)\big)K_\si(x,y)\,d\eta(y)+f(x)B_{\si}(x),
\end{equation}
where  $K_\si(x,y)$ is a symmetric (in the sense that $K_\si(x,y)=K_\si(y,x)$), positive kernel not necessarily known  explicitly and  $a_{\si}$ is a positive constant, that depends on the kernel  and the dimension of the underlying  manifold $X$. Furthermore, $B_{\si}(x)$ is a continuous non-negative bounded function.

The next expression follows from the integral representation. We will also assume that the kernel $K_\si(x,y)$ is measurable and $\int_X\int_X|x-y|^2K_{\si}(x,y)\,dx\,dy<\infty$ to allow the interchange of the order of integration in the following computations.

\begin{lem}
\label{lem:bilinear}
Let $0<\si<1$ and assume that the representation \eqref{eq:integral-rep} is valid.  Then, for all $f\in C_0^{\infty}(X)$
\begin{equation}
\label{eq:bilinear-form}
\langle\Lambda_\si f,f\rangle = \frac{a_{\si}}{2}\int_X\int_X|f(x)-f(y)|^2K_\si(x,y)\,d\eta(x)\,d\eta(y)
+B_{\si}(x)\langle f,f\rangle.
\end{equation}

\end{lem}
\begin{proof}
From the integral representation \eqref{eq:integral-rep} and Fubini we have
$$
\langle\Lambda_\si f,f\rangle=\int_X\Lambda_\si f(x)\overline{f(x)}\,d\eta(x)
= a_{\si}\int_X\int_X\big(f(x)-f(y)\big)\overline{f(x)}K_\si(x,y)\,d\eta(x)\,d\eta(y)
+B_{\si}(x)\langle  f,f\rangle.
$$
Since the kernel $K_\si(x,y)$ is symmetric, the above also equals to
$$
\langle\Lambda_\si f,f\rangle=\int_X\Lambda_\si f(x)\overline{f(x)}\,d\eta(x)
= a_{\si}\int_X\int_X\big(f(y)-f(x)\big)\overline{f(y)}K_\si(x,y)\,d\eta(x)\,d\eta(y)
+B_{\si}(x)\langle  f,f\rangle.
$$
Adding them up we get the announced statement.
\end{proof}

Even though we have proved the above lemma under the assumption that $ f  \in C_0^\infty(X)$, in order to make use of the representation \eqref{eq:bilinear-form} it is necessary to assume its validity for a suitable Sobolev space which depends on the operator $ A $ and $ \si$.  Let us denote such a Sobolev space by $ W_A^\si(X) $ and assume that \eqref{eq:integral-rep} and hence the result of Lemma \ref{lem:bilinear} are valid for all $ f \in W_A^\si(X)$.

\subsection{The ground state representation and a Hardy-type inequality}

Let $0<\si<1$. For a suitable positive function $w_\si(x)$  to be specified  later, let us set
$$
\mathcal{H}_\si[f]:=\langle\Lambda_\si f,f\rangle-C_{\si}\int_X \frac{\widetilde{w}_\si(x)}{w_\si(x)}|f(x)|^2\,d\eta(x),
$$
where $C_{\si}$ is a positive constant and $\widetilde{w}_\si(x) := \Lambda_\si w_\si(x)$. If we can show that $\mathcal{H}_\si[f]$ is nonnegative, then we get a version of  Hardy's inequality for the operator $ \Lambda_\si$.

\begin{thm}
\label{thm:groundstate}
Let $0<\si<1$ and assume the validity of Lemma \ref{lem:bilinear} for all $ f \in W_A^\si(X)$.  Given $f \in C_0^\infty(X) $ and a  positive function $ w_\si \in W_A^\si(X) $ set $g(x)=f(x)(w_\si(x))^{-1}$. Then
$$
\mathcal{H}_\si[f]
=\frac{a_{\si}}{2}\int_X\int_X|g(x)-g(y)|^2K_\si(x,y)w_\si(x)w_\si(y)\,d\eta(x)\,d\eta(y),
$$
where $a_{\si}$ is the positive constant in Lemma \ref{lem:bilinear}.
\end{thm}
\begin{proof}
By polarizing the expression in Lemma \ref{lem:bilinear} we get, for any $F,G\in C_0^{\infty}(X)$,
\begin{equation}
\label{eq:polarization}
\langle\Lambda_\si F,G\rangle= \frac{a_{\si}}{2}\int_X\int_X\big(F(x)-F(y)\big)
\overline{\big(G(x)-G(y)\big)}K_\si(x,y)\,d\eta(x)\,d\eta(y)+B_{\si}(x)\langle F,G\rangle.
\end{equation}
Now, we take $G(x)=w_\si(x)$ and $F(x)=|f(x)|^2w_\si(x)^{-1}$. After simplification, the right hand side of \eqref{eq:polarization} becomes
$$
\frac{a_{\si}}{2}\int_X\int_X\bigg(|f(x)-f(y)|^2-
 \Big|\frac{f(x)}{w_\si(x)}-\frac{f(y)}{w_\si(y)}\Big|^2w_\si(x)w_\si(y)\bigg)
 K_\si(x,y)\,d\eta(x)\,d\eta(y)+B_{\si}(x)\langle f,f\rangle.
$$
On the other hand,  using the fact that $ A_\si $ is self-adjoint and recalling the definition of $ \widetilde{w}_\si$ we could write the left hand side of \eqref{eq:polarization} as
$$
C_{\si}\int_{X}\frac{\widetilde{w}_\si(x)}{w_\si(x)}|f(x)|^2\,d\eta(x).
$$
Thus we have
\begin{multline*}
C_{\si}\int_{X}\frac{\widetilde{w}_\si(x)}{w_\si(x)}|f(x)|^2\,d\eta(x)= \frac{a_{\si}}{2}\int_X\int_X|f(x)-f(y)|^2K_\si(x,y)\,d\eta(x)\,d\eta(y)+B_{\si}(x)\langle f,f\rangle\\
-\frac{a_{\si}}{2}\int_X\int_X \Big|g(x)-g(y)\Big|^2 K_\si(x,y)w_\si(x)w_\si(y)
\,d\eta(x)\,d\eta(y).
\end{multline*}
The latter identity and Lemma \ref{lem:bilinear} proves the theorem.
\end{proof}

Consequently, Theorem \ref{thm:groundstate} leads to the following Hardy-type inequality.

\begin{cor}
\label{cor:Hardyineq}
Let $0<\si<1$. Then
$$
\langle\Lambda_\si f,f\rangle \ge C_{\si}\int_X|f(x)|^2\frac{\widetilde{w}_\si(x)}{w_\si(x)}\,d\eta(x),
$$
for all functions $f\in C_0^{\infty}(X)$.
\end{cor}

\begin{rem}
In the proof of Theorem \ref{thm:groundstate} we have assumed that  the action of the operator $\Lambda_{\si}$ on the  weight $w_{\si}$ can be calculated. In several cases, this can be done via the spectral theorem but it is important that the resulting function $ \widetilde{w}_\si $ is explicit. Though this can be done in several particular cases, we do not have a general result guaranteeing such a simplification. Actually, this step is crucial, and so the choice of such weight a function is dictated by the requirement that $ \widetilde{w}_\si $ can be computed explicitly. In most cases, we use a function $ w_\si $ related to  the fundamental solution of the operator $\Lambda_{\si}.$
\end{rem}

\begin{rem}
Corollary \ref{cor:Hardyineq} is not necessarily  a sharp Hardy inequality, but it is an intermediate step towards a sharp Hardy inequality of the type \eqref{eq:hardy}. The task boils down to minimizing the function $a_{\si}\frac{\widetilde{w}_\si(x)}{w_\si(x)}$, keeping a careful track of the constants involved. In several specific cases this function can be explicitly computed, so the sharp Hardy inequality, modulo verifications on the convergence of the integrals, is directly obtained from Theorem \ref{thm:groundstate} (see \cite[Proposition 4.1]{FLS} for the Euclidean setting or \cite[Theorems 5.2 and 5.4]{RT} for the Heisenberg group).
\end{rem}

The definitions and the context above are rather general. We remind that in the following sections we will develop the procedure in the particular cases of Laguerre and Dunkl--Hermite settings. We will deal with the conformally invariant fractional $\si$-th powers of the corresponding operators $A$, which we denote by $\Lambda_{\si}$. Such operators will be defined by using spectral decomposition.

We have just outlined  a general technique to get Hardy inequalities. This technique was first developed in the Euclidean context in \cite{FLS}, and improved and generalized, always related to the Euclidean Laplacian, in \cite{FS}. Such a method was for the first time  adapted to get Hardy inequalities in the Heisenberg group in  \cite{RT}. Our main purpose in this paper is to show  the robustness of this technique by proving   Hardy inequalities related to operators which are not translation invariant.  In the subsequent sections we will apply this procedure to get Hardy inequalities for the Laguerre operator, the harmonic oscillator, and the Dunkl--Hermite operator.

\section{A Hardy inequality for the fractional Laguerre operator}
\label{sec:HardyLaguerre}

\subsection{The fractional powers of the Laguerre operator}
\label{subsec:fpowersL}
We will introduce first some facts related to Laguerre functions. From now on, let $\al>-1/2$. For $r,s>0$, we define the \textit{Laguerre translation} $\mathcal{T}_r^{\al}f$ of a function $f$ on $(0,\infty)$ by
\begin{equation*}
\mathcal{T}_r^{\al}f(s)=\frac{\Gamma(\al+1)2^{\al}}{\sqrt{2\pi}}\int_0^{\pi}f
\big((r^2+s^2+2rs\cos\theta)^{1/2}\big)J_{\al-1/2}(rs\sin\theta)
(rs\sin\theta)^{-(\al-1/2)}(\sin\theta)^{2\al}\,d\theta,
\end{equation*}
where $J_{\nu}$ is the Bessel function of order $\nu$. If $f$ and $g$ are functions defined on $(0,\infty)$, the \textit{Laguerre convolution} $f\ast_{\al}g$ is given by
\begin{equation}
\label{eq:LagConv}
f\ast_{\al}g=\int_0^{\infty}\mathcal{T}_r^{\al}f(s)g(s)s^{2\al+1}\,ds,
\end{equation}
see \cite[Chapter 6]{ST} and references therein for the definitions and properties of Laguerre translation and Laguerre convolution.

For $r\in (0,\infty)$, we define the Laguerre functions of type $\al$ as
\begin{equation*}
\varphi_n^{\al}(r):=L_n^{\al}(r^2)e^{-r^2/2}, \quad n=0,1,\ldots,
\end{equation*}
where $L_{n}^{\al}$ are the Laguerre polynomials of order $\al$. The functions $\{\varphi_n^{\al}\}_{n=0}^{\infty}$ form an orthogonal basis for $ L^2((0,\infty),d\mu_{\al})$, where $d\mu_{\al}(r)=r^{2\al+1}$. When we apply the Laguerre translation to Laguerre functions we have the indentity (see \cite[(6.1.28)]{ST})
\begin{equation}
\label{eq:Lagtransvarphi}
\mathcal{T}_r^\alpha \varphi_n^\alpha(s)=\frac{n!}{(\alpha+1)_n}
\varphi_n^\alpha(r)\varphi_n^\alpha(s),\qquad \alpha>-1/2,
\end{equation}
with the notation $(x)_n=x(x-1)(x-2)\cdots (x-n+1)$ standing for the  Pochhammer symbol.

As stated in the introduction, we will deal with the Laguerre differential operator given by
\begin{equation}\label{eq:differential operator}
L_{\al}=-\frac{d^2}{dr^2}+r^2-\frac{2\al+1}{r}\frac{d}{dr},
\end{equation}
which is symmetric on  $ L^2((0,\infty), d\mu_\al) $ taken with respect to the measure $ d\mu_\al$. The functions $ \varphi_n^\al $ are eigenfunctions of the differential
operator \eqref{eq:differential operator}. Indeed, for $ n =0, 1, 2,....$
$$ L_\al \varphi_n^\al = (4n+2\al+2) \varphi_n^\al.$$  The family of
functions $\psi_n^{\al}$, given by
$$
\psi_n^{\al}(r) = \bigg(\frac{2\Gm(n+1)}{\Gm(n+\al+1)} \bigg)^{1/2}
\varphi_n^{\al}(r), \quad r>0,
$$
forms an orthonormal basis for $ L^2((0,\infty), d\mu_\al).$ The Laguerre expansion of a function $ f \in L^2((0,\infty), d\mu_\al),$  namely the expansion
$$  f = \sum_{n=0}^\infty \bigg(\frac{2\Gm(n+1)}{\Gm(n+\al+1)} \bigg) (f, \varphi_n^\al) \varphi_n^\al $$ can be written in a compact form in terms of Laguerre convolution.
\begin{lem}
For a function $f \in L^2((0,\infty),d\mu_\al) $  we have
$$
f=\frac{2}{\Gamma(\al+1)}\sum_{n=0}^{\infty}f\ast_{\al}\varphi_n^{\al}
$$
where the series converges in norm. In particular,
\begin{equation}
\label{eq:convoLaguerres}
\delta_{nj}\varphi_n^{\al}=\frac{2}{\Gamma(\al+1)}\varphi_n^{\al}\ast_{\al}\varphi_j^{\al}.
\end{equation}
\end{lem}
\begin{proof}
Let us define $P_n^{\al}$ to be the projection onto the $n$th eigenspace
$$
P_n^{\al}f(r)=\psi_n^{\al}(r)\int_0^{\infty}f(s)\psi_n^{\al}(s)\,d\mu_{\al}(s).
$$
Therefore, by \eqref{eq:Lagtransvarphi},
\begin{align*}
f(r)&=\sum_{n=0}^{\infty}P_n^{\al}f(r)=2\sum_{n=0}^{\infty}\int_0^{\infty}f(s)\frac{\Gamma(n+1)}
{\Gamma(n+\al+1)}
\varphi_n^{\al}(r)\varphi_n^{\al}(s)\,d\mu_{\al}(s)\\
&=\frac{2}{\Gamma(\al+1)}\sum_{n=0}^{\infty}\int_0^{\infty}
f(s)\mathcal{T}_r^{\al}\varphi_n^{\al}(s)\,
d\mu_{\al}(s)=\frac{2}{\Gamma(\al+1)}\sum_{n=0}^{\infty}f\ast_{\al}\varphi_n^{\al}(r),
\end{align*}
where we used \cite[Proposition 6.1.1]{ST} in the fourth equality above. Finally, by computing $P_n^{\al}\varphi_j^{\al}$, and using that $P_n^{\al}f(r)= \frac{2}{\Gamma(\al+1)} f\ast_{\al}\varphi_n^{\al}(r)$ (which is easily deduced from the equalities above), we obtain \eqref{eq:convoLaguerres}.
\end{proof}

Thus the  spectral decomposition of the Laguerre operator is given by
$$
L_{\al}f(r)=\frac{2}{\Gamma(\al+1)}\sum_{n=0}^{\infty}(4n+2\al+2) f\ast_{\al}\varphi_n^{\al}(r).
$$
Therefore, a natural way to define  fractional powers of the Laguerre operator is via the spectral decomposition:
$$
L_{\al}^{\si}f(r)=\frac{2}{\Gamma(\al+1)}\sum_{n=0}^{\infty}(4n+2\al+2)^{\si} f\ast_{\al}\varphi_n^{\al}(r).
$$
Nevertheless, it will be more convenient to work with the following modified fractional powers $L_{\al,\si}$. For $0\le \si \le 1$ we define $L_{\al,\si}$ by
\begin{equation}
\label{eq:spectralLsi}
L_{\al,\si}f(r)=\frac{2}{\Gamma(\al+1)}\sum_{n=0}^{\infty}4^{\si}
S_n^{\alpha,\si}f\ast_{\al}\varphi_n^{\al}(r),
\end{equation}
where $S_n^{\alpha,\si}$ is as in \eqref{eq:Salpha}.
In short, the above means that $L_{\al,\si}$ is the operator
$$
L_{\al,\si}:=4^{\si}\frac{\Gamma\big(\frac{L_{\al}}{4}+\frac{1+\si}{2}\big)}
{\Gamma\big(\frac{L_{\al}}{4}+\frac{1-\si}{2}\big)}
$$
corresponding to the spectral multiplier $4^\si S_n^{\alpha,\si}$ with $S_n^{\alpha,\si}$ defined in \eqref{eq:Salpha}.
The motivation for this definition goes back, for instance, to \cite[(1.33)]{BFM}. Observe that $L_{\al,1}= L_{\al}$. This operator has an explicit fundamental solution, see Subsection \ref{subsec:CH}, and this fact makes it more suitable than $L_{\al}^{\si}$, whose fundamental solution cannot be written down explicitly. Moreover, by using Stirling's formula for the Gamma function, one can readily see that $L_{\al,\si}=U_{\al,\si}L_{\al}^{\si}$, where $U_{\al,\si}$ is a bounded operator on $L^2((0,\infty), d\mu_{\al})$.

\subsection{The Laguerre heat semigroup}

We will use the language of semigroups to get a suitable integral representation for $L_{\al,\si}$. The heat semigroup
related to $L_{\al}$ is  defined on $ L^2((0,\infty),  d\mu_\al) $ by
\begin{equation}
\label{eq:heatsemigroupLag}
e^{-tL_{\al}}f:=T_{\al,t}f=\frac{2}{\Gamma(\al+1)}\sum_{n=0}^{\infty}e^{-t(4n+2\al+2)}
f\ast_{\al}\varphi_n^{\al},\quad t>0.
\end{equation}
If we define
$$
q_{t,\al}(r)=\frac{2}{\Gamma(\al+1)}\sum_{n=0}^{\infty}e^{-(4n+2\al+2)t}\varphi_n^{\al}(r),
$$
then the generating function identity for $\varphi_n^{\al}$ (see \cite[(1.4.24)]{STH}) immediately gives
\begin{equation}
\label{eq:heatkernelLag}
q_{t,\al}(r)=\frac{1}{2^{\al}\Gamma(\al+1)}(\sinh 2t)^{-\al-1}e^{-(\coth 2t)r^2/2},
\end{equation}
and we can write
$$
e^{-tL_{\al}}f=f\ast_{\al}q_{t,\al}.
$$
Let us denote by $I_{\al}$ the modified Bessel function of the first
kind and order $\al$, see \cite[Chapter 5, Section 5.7]{Lebedev}. Then, from the definition of $q_{t,\al}(r)$, the identity \eqref{eq:Lagtransvarphi}, and another generating function identity for $\varphi_n^{\alpha}(r)$ (see \cite[p. 83]{ST})
\begin{equation}
\label{eq:generating}
\sum_{n=0}^{\infty}\frac{\Gamma(n+1)}{\Gamma(n+\alpha+1)}\varphi_n^{\alpha}(r)
\varphi_n^{\alpha}(s)w^{2n}=(1-w^2)^{-1}(rsw)^{-\alpha}\exp\Big\{-\frac{1}{2}\Big (\frac{1+w^2}{1-w^2}\Big )
 (r^2+s^2)\Big\}I_{\alpha}\Big(\frac{2wrs}{1-w^2}\Big),
\end{equation}
we obtain readily the following lemma.
\begin{lem}
\label{lem:Lagtranskernel}
Let $\al>-1/2$, we have that
$$
\mathcal{T}^{\al}_r q_{t,\al}(s)=\frac{e^{-\frac{\coth 2t}{2}(r^2+s^2)}}{(rs)^\al \sinh 2t}I_\al\left(\frac{rs}{\sinh 2t}\right).
$$
\end{lem}
The  Laguerre heat semigroup is not conservative, i.e. $ T_{\al,t}1 \neq 1.$  However, it is easy to calculate and estimate $ T_{\al,t}1 $ which is the content of the following lemma.
\begin{lem}
\label{lem:conservative}
Let $\al>-1/2$, we have that
$$
T_{\al,t}1(r)=\frac{1}{(\cosh 2t)^{\al+1}}e^{-\frac{\tanh 2t}{2}r^2}\le1.
$$
\end{lem}
\begin{proof}
By \eqref{eq:LagConv} and the previous lemma
\begin{align*}
T_{\al,t}1(r)&=1\ast_{\al}q_{t,\al}(r)
=\int_0^{\infty}\mathcal{T}_r^{\al}q_{t,\al}(s)s^{2\al+1}\,ds\\
&=\frac{e^{-\frac{\coth 2t}{2}r^2}}{r^\al \sinh 2t}\int_0^{\infty}e^{-\frac{\coth 2t}{2}s^2}I_\al\left(\frac{rs}{\sinh 2t}\right)s^{\al+1}\,ds.
\end{align*}
Now, the result is a consequence of the identity (see \cite[(2.15.5.4)]{Prudnikov2})
\[
\int_0^\infty e^{-px^2}I_\alpha(cx) x^{\alpha+1}\, dx=\frac{c^\alpha e^{c^2/(4p)}}{(2p)^{\alpha+1}}, \qquad p>0,\quad \alpha>-1.
\]
\end{proof}
\subsection{Integral representation for $L_{\al,\si}$}

In order to obtain an integral representation for the operator $L_{\al,\si}$ we first prove a numerical identity. We will use the identity (see \cite[p. 382, 3.541.1]{GR})
\begin{equation}
\label{eq:GR0}
\int_0^{\infty}e^{-\mu t}\sinh^\nu\beta t\,dt=\frac{1}{2^{\nu+1}}
\frac{\Gamma\big(\frac{\mu}{2\beta}-\frac{\nu}{2}\big)\Gamma(\nu+1)}
{\Gamma\big(\frac{\mu}{2\beta}+\frac{\nu}{2}+1\big)},
\end{equation}
which is valid for $\operatorname{Re}\beta>0$, $\operatorname{Re}\nu>-1$, $\operatorname{Re}\mu>\operatorname{Re}\beta\nu$.
\begin{lem}
\label{lem:numerical1}
Let $0<\si<1$, and $\lambda\in \R$ such that $\lambda+2\si>-2$. Then,
$$
2^\si|\Gamma(-\si)|\frac{\Gamma\big(\frac{\lambda}{4}+\frac{1+\si}{2}\big)}
{\Gamma\big(\frac{\lambda}{4}+\frac{1-\si}{2}\big)}=\int_0^{\infty}(\cosh 2t-1)(\sinh 2t)^{-\si-1}\,dt+\int_0^{\infty}(1-e^{-t\lambda})(\sinh 2t)^{-\si-1}\,dt.
$$
\end{lem}
\begin{proof}

We take $\nu=-\si$, $\beta=2$, and $\mu=\lambda+2$ in \eqref{eq:GR0}. Thus
$$
\int_0^{\infty}e^{-(\lambda+2) t}(\sinh 2t)^{-\si}\,dt=\frac{2^{\si-1}\Gamma\big(\frac{\lambda+2}{4}+\frac{\si}{2}\big)
\Gamma(1-\si)}
{\Gamma\big(\frac{\lambda}{4}+\frac{1-\si}{2}+1\big)}
$$
or, equivalently,
\begin{equation}
\label{eq:id1}
\big(\frac{\lambda}{2}+1-\si\big)\int_0^{\infty}e^{-(\lambda+2) t}(\sinh 2t)^{-\si}\,dt=\frac{2^{\si}\Gamma\big(\frac{\lambda+2}{4}+\frac{1+\si}{2}\big)
\Gamma(1-\si)}
{\Gamma\big(\frac{\lambda}{4}+\frac{1-\si}{2}\big)}.
\end{equation}
Moreover, an integration by parts yields
\begin{multline*}
(\lambda+2)\int_0^{\infty}e^{-(\lambda+2) t}(\sinh 2t)^{-\si}\,dt=\int_0^{\infty}\frac{d}{dt}\big(1-e^{-(\lambda+2)t}\big)(\sinh 2t)^{-\si}\,dt\\
=2\si\int_0^{\infty}\big(1-e^{-(\lambda+2)t}\big)(\sinh 2t)^{-\si-1}\cosh 2t\,dt.
\end{multline*}
Combining this with \eqref{eq:id1}, and noticing that $\frac{\Gamma(1-\si)}{\si}=|\Gamma(-\si)|$, we deduce
\begin{align*}
\frac{2^{\si}\Gamma\big(\frac{\lambda}{4}+\frac{1+\si}{4}\big)|\Gamma(-\si)|}
{\Gamma\big(\frac{\lambda}{4}+\frac{1-\si}{2}\Big)}&=\int_0^{\infty}
\big(\big(1-e^{-(\lambda+2)t}\big)(\sinh 2t)^{-\si-1}\cosh 2t-e^{-(\lambda+2) t}(\sinh 2t)^{-\si}\big)\,dt\\
&=\int_0^{\infty}\big(\cosh 2t-e^{-(\lambda+2)t}(\cosh 2t+\sinh 2t)\big)(\sinh 2t)^{-\si-1}\,dt\\
&=\int_0^{\infty}\big(\cosh 2t-e^{-\lambda t}(\sinh 2t)^{-\si-1}\big)\,dt\\
&=\int_0^{\infty}(\cosh 2t-1)(\sinh 2t)^{-\si-1}\,dt+\int_0^{\infty}(1-e^{-t\lambda})(\sinh 2t)^{-\si-1}\,dt.
\end{align*}
The proof is complete.
\end{proof}

Define the functions
$$
\mathcal{K}_{\al,\si}(r,s):=\int_0^{\infty}\mathcal{T}_r^{\al}q_{t,\al}(s)\frac{dt}{(\sinh 2t)^{\si+1}}, \quad F_{\al,\si}(r):=\int_0^{\infty}\big(1-T_{\al,t}1(r)\big)\frac{dt}{(\sinh 2t)^{\si+1}}.
$$
The integral defining the function $F_{\al,\si}(r)$ has to be taken  in the sense of Bochner, i.e. as the integral of the $L^2((0,\infty),d\mu_\al)$ valued function $t\to (1-1\ast_{\al}q_{t,\al})$. Observe that this function  is non-negative, by Lemma \ref{lem:conservative}. The expression in Lemma \ref{lem:Lagtranskernel} allows us to deduce that $\mathcal{K}_{\al,\si}$ is a positive function. Let us call $E_{\si}:=\frac{2^{\si}}{|\Gamma(-\si)|}\int_0^{\infty}(\cosh 2t-1)(\sinh 2t)^{-\si-1}\,dt$. The constant $E_{\si}$ can be explicitly computed, see Remark~\ref{rem:Es}. We are ready to prove the integral representation for $L_{\al,\si}$.

\begin{prop}
\label{prop:integralLaguerre}
Let $0<\si<1$ and $f\in C_0^{\infty}(0,\infty)$. Then we have the following pointwise representation
$$
L_{\al,\si}f(r)=\frac{2^{\si}}{|\Gamma(-\si)|}
\int_0^{\infty}(f(r)-f(s))\mathcal{K}_{\al,\si}(r,s)\,d\mu_{\al}(s)
+f(r)\Big(E_{\si}+\frac{2^{\si}}{|\Gamma(-\si)|}F_{\al,\si}(r)\Big).
$$
\end{prop}
\begin{proof}

In the identity of Lemma \ref{lem:numerical1} we take $\lambda=4n+2\al+2$, multiply both sides by $4^{\si}\frac{2}{\Gamma(\al+1)}f\ast_{\al}\varphi_n^{\al}$ and  then sum over $n$. Thus, with the constants $S_n^{\alpha,\sigma}$ as in \eqref{eq:Salpha}, and taking into account \eqref{eq:heatsemigroupLag},
\begin{align*}
L_{\al,\si}f(r)&=\frac{2}{\Gamma(\al+1)}\sum_{n=0}^{\infty}4^{\si}S_n^{\alpha,\sigma}
f\ast_{\al}\varphi_n^{\al}\\
&=E_{\si}f(r)+\frac{2^{\si}}{|\Gamma(-\si)|}\int_0^{\infty}
\big(f(r)-f\ast_{\al}q_{t,\al}(r)\big)
(\sinh 2t)^{-\si-1}\,dt,
\end{align*}
and this last integral has also to be interpreted  as the Bochner integral of the $L^2((0,\infty),d\mu_\al)$ valued function $t\to f-f\ast_{\al}q_{t,\al}$.
We have
\begin{align*}
f(r)-f\ast_{\al}q_{t,\al}(r)&=f(r)-\int_0^{\infty}
\mathcal{T}_{r}^{\al}f(s)q_{t,\al}(s)\,d\mu_{\al}(s)
=f(r)-\int_0^{\infty}\mathcal{T}_{r}^{\al}q_{t,\al}(s)f(s)
\,d\mu_{\al}(s)\\
&=f(r)-f(r)1\ast_{\al}q_{t,\al}(r)+f(r)\int_0^{\infty}
\mathcal{T}_{r}^{\al}q_{t,\al}(s)\,d\mu_{\al}(s)
-\int_0^{\infty}\mathcal{T}_{r}^{\al}q_{t,\al}(s)f(s)\,d\mu_{\al}(s)\\
&=f(r)\big(1-1\ast_{\al}q_{t,\al}(r)\big)+\int_0^{\infty}
\mathcal{T}_{r}^{\al}q_{t,\al}(s)
\big(f(r)-f(s)\big)
\,d\mu_{\al}(s).
\end{align*}
Therefore,
\begin{align}
\label{eq:integral}
\notag\int_0^{\infty}\big(f(r)-f\ast_{\al}q_{t,\al}(r)\big)\frac{dt}{(\sinh 2t)^{\si+1}}&=f(r)\int_0^{\infty}\big(1-1\ast_{\al}q_{t,\al}(r)\big)\frac{dt}{(\sinh 2t)^{\si+1}}\\
&\quad +\int_0^{\infty}\Big[\int_0^{\infty}T_{r}^{\al}q_{t,\al}(s)
\big(f(r)-f(s)\big)
\,d\mu_{\al}(s)\Big]\frac{dt}{(\sinh 2t)^{\si+1}}.
\end{align}
We have to justify the application of Fubini's theorem in the integral \eqref{eq:integral}. In the following lines, we will denote by $C_{\al}$ (or simply $C$) a constant depending on $\al$ (or even independent of $\al$) that may vary at any occurrence. By Lemma \ref{lem:Lagtranskernel} and  applying the change of variable due to S. Meda
$$
t(\xi)=\frac12\log\frac{1+\xi}{1-\xi}, \qquad t\in (0,\infty),\quad \xi \in (0,1),
$$
we have, see \cite[Proposition 4.2]{CR}, that
$$
\mathcal{T}^{\al}_r q_{t(\xi),\al}(s))\le C(rs)^{-\al}\Big(\frac{1-\xi^2}{2\xi}\Big)^{1/2}
e^{-\frac14\big(\frac{|r-s|^2}{\xi}+\xi|r+s|^2\big)},
$$
and the expression in the right hand side of the inequality is precisely the heat kernel associated to the harmonic oscillator, see \cite{ST}. Moreover, such change of parameters also yields
$$
\frac{dt}{(\sinh 2t)^{\si+1}}=\frac{(1-\xi^2)^{\si}}{(2\xi)^{\si+1}}\,d\xi, \qquad t\in (0,\infty), \quad \xi \in (0,1).
$$
Notice that
$$
\frac{(1-\xi^2)^{\si}}{(2\xi)^{\si+1}}\,d\xi\sim \frac{d\xi}{\xi^{\si+1}}, \quad \xi\sim 0, \qquad \frac{(1-\xi^2)^{\si}}{(2\xi)^{\si+1}}\,d\xi\sim (1-\xi^2)^{\si}\,d\xi,\quad \xi\sim 1.
$$
With the ingredientes just collected, the absolute convergence of the integral in \eqref{eq:integral} can be concluded following exactly the same reasoning as in \cite[Proof of Theorem 5.7, p. 2118]{STo}, where the authors prove a pointwise formula for the fractional harmonic oscillator. Finally, we also deduce that $F_{\al,\si}\in C^{\infty}(\R^d)$ by following an analogous reasoning as in \cite[Lemma 5.11]{STo}.
\end{proof}
\begin{rem}
\label{rem:Es}
The constant $E_{\si}:=\frac{2^{\si}}{|\Gamma(-\si)|}\int_0^{\infty}(\cosh 2t-1)(\sinh 2t)^{-\si-1}\,dt$ can be explicitly computed. Indeed, with Meda's change of variable, the integral boils down to
$$
2^{-\si}\int_0^{1}(1-\xi^2)^{\si-1}(\xi^2)^{(1-\si)/2}\,d\xi=2^{-\si-1}\int_0^{1}(1-w)^{\si-1}w^{-\si/2}\,dw=2^{-\si-1}B(1-\si/2,\si),
$$
where $B(x,y)$ is the Beta function. Then, by writing the latter in terms of Gamma functions, and in view of the duplication formula, we get
$$
E_{\si}=\frac{2^{\si}}{|\Gamma(-\si)|}2^{-\si-1}B(1-\si/2,\si)=\frac{\Gamma(1-\si/2)
\Gamma(\si)}{2\Gamma(1+\si/2)|\Gamma(-\si)|}=\frac{(-\si/2)\Gamma(-\si/2)\Gamma(\si)}
{2(\si/2)\Gamma(\si/2)|\Gamma(-\si)|}=4^{\si}\frac{\Gamma\big(\frac{1+\si}{2}\big)}
{\Gamma\big(\frac{1-\si}{2}\big)}.
$$
\end{rem}

Following the general procedure showed in Section \ref{sec:generalMethod}, we continue with the integral representation for $\langle L_{\al,\si}f,f\rangle$.

\begin{prop}
\label{prop:integRep}
Let $0<\si<1$ and $\al>-1$. Then, for all $f\in C_0^{\infty}(0,\infty)$
\begin{multline*}
\langle L_{\al,\si}f,f\rangle=\frac{2^{\si-1}}{|\Gamma(-\si)|}
\int_0^{\infty}\int_0^{\infty}|f(r)-f(s)|^2\mathcal{K}_{\al,\si}(r,s)\,d\mu_{\al}(r)d\mu_{\al}(s)\\
+\frac12\Big(E_{\si}+\frac{2^{-\si}}{|\Gamma(-\si)|}\Big)
\int_0^{\infty}f(r)^2F_{\al,\si}(r)\,d\mu_{\al}(r).
\end{multline*}
\end{prop}
\begin{proof}
Proceed as in Lemma \ref{lem:bilinear}, by using the integral representation in Proposition \ref{prop:integralLaguerre}.
\end{proof}

\subsection{Spectral analysis for $L_{\al,\si}$ and the fundamental solution}
\label{subsec:CH}
In this subsection we figure out a suitable weight $w_{\si}$ to be plugged in the computation of the ground state representation, and find out the action of the operator $L_{\al,\si}$ on such a weight. We may look at the situation in the Laguerre setting as a particular case of the Heisenberg setting, in which the functions considered are radial in the first variable, and the convolution is just the Laguerre convolution. So we find hints in the arguments developed in \cite[Section 3]{RT}.

For $\dl>0$, let $u_{\si,\delta}(r,t)$ be the function defined on $(0,\infty)\times \R$ by
$$
u_{\si, \dl}(r,t):=\big((\dl+r^2)^2+t^2\big)^{-(\al+2+\si)/2}
$$
(observe that the function above is the one defined by \cite[p. 530]{CH}, see also \cite[(3.1)]{RT}, but radial in the first variable, and with suitable parameters).

Let $f^\la$ stand for the inverse Fourier transform of a function $f\in L^1((0,\infty)\times \R)$ in the second variable, i.e., for each $\la\in \R$, we set
$$
f^\la(r):=\frac{1}{\sqrt{2\pi}}\int_{-\infty}^{\infty}e^{i\la t}f(r,t)\,dt.
$$
For $ a, b \in \R^+ $ and $ c \in \R $ we define
\begin{equation*}
L(a,b,c) = \int_{0}^\infty e^{-a(2x+1)} x^{b-1}\big(1+x\big)^{-c} dx.
\end{equation*}
In the following lemma we compute the Laguerre expansion of  $u_{\si,\dl}^{\la}(r)$.
\begin{lem}
\label{lem:CowHag}
Let $0<\si<1$, $\dl>0$, $\al>-1/2$, and $\la\in  \R\setminus\{0\}$. We have
$$
u_{\si,\dl}^{\la}(r)=|\la|^{\al+1}
\sum_{n=0}^{\infty}c_{n,\dl}^{\la}(\si)\varphi_n^{\al}((2|\la|)^{1/2} r),
$$
where the coefficients $c_{n,\dl}^{\la}(\si)$ are given by
\begin{equation}
\label{eq:coefficient}
c_{n,\dl}^\lambda(\si)  = \frac{2\pi |\lambda|^\si}{\Gamma\big((\al+2+\si)/2\big)^2}L\Big( \dl|\lambda|, \frac{4n+2\al+2}{4}+\frac{1+\si}{2},\frac{4n+2\al+2}{4}+\frac{1-\si}{2}\Big).
\end{equation}
\end{lem}
\begin{proof}
The proof is already present implicitly  in the proof of \cite[Proposition 3.2]{RT}.
\end{proof}

It happens that the Fourier transform of the function $u_{\si,\dl}(r,t)$ in the $t$-variable can be evaluated in terms of a Macdonald's function $K_{\nu}$. Let us recall the definition of the weight $w_{\al,\si}(r)$, already given in \eqref{eq:weight}
\begin{equation*}
w_{\al,\si}^{\dl}(r):=c_{\al,\si}(\dl+r^2)^{-(\al+1+\si)/2}K_{(\al+1+\si)/2}\big((\dl+r^2)/2\big),
\end{equation*}
where $c_{\al,\si}$ is the constant in \eqref{eq:calsi}. Indeed, we have the following.
\begin{prop}
\label{eq:LagseriesK}
Let $0<\si<1$, $\dl>0$, and $\al>-1/2$. Then
$$
w_{\al,\si}^{\dl}(r)
=\sum_{n=0}^{\infty}c_{n,\dl}^{1/2}(\si)\varphi_n^{\al}(r).
$$
\end{prop}
\begin{proof}
In \cite[2.5.6. 4., p. 390]{Prudnikov1}, we find the formula
$$
\int_0^{\infty}\frac{\cos bx}{(x^2+z^2)^{\rho}}\,dx=\Big(\frac{2z}{b}\Big)^{1/2-\rho}\frac{\sqrt{\pi}}{\Gamma(\rho)}
K_{1/2-\rho}(bz),
$$
which is valid for $b, \Re \rho, \Re z>0$. This formula gives, assuming $ \lambda > 0,$
$$
\int_{-\infty}^{\infty}e^{i\la t}u_{\si}(r,t)\, dt=2\frac{\sqrt{\pi}2^{-(\al+1+\si)/2}}{\Gamma[(\al+2+\si)/2]}
\Big(\frac{\dl+r^2}{\la}\Big)^{-(\al+1+\si)/2}K_{-(\al+1+\si)/2}(\la(\dl+r^2)).
$$
Thus, by Lemma \ref{lem:CowHag} we have the following Laguerre expansion of $K_{-(\al+1+\si)/2}(\la(\dl+r^2))$:
$$
2\frac{\sqrt{\pi}2^{-(\al+1+\si)/2}}{\Gamma[(\al+2+\si)/2]}
\Big(\frac{\dl+r^2}{\la}\Big)^{-(\al+1+\si)/2}K_{-(\al+1+\si)/2}(\la(\dl+r^2))=\la^{\al+1}
\sum_{n=0}^{\infty}c_{n,\dl}^{\la}(\si)\varphi_n^{\al}((2\la)^{1/2} r).
$$
By taking $\la=1/2$ we get
$$
\frac{\sqrt{\pi}2^{1-\si}}
{\Gamma[(\al+2+\si)/2]}(\dl+r^2)^{-(\al+1+\si)/2}K_{-(\al+1+\si)/2}\big((\dl+r^2)/2\big)
=\sum_{n=0}^{\infty}c_{n,\dl}^{1/2}(\si)\varphi_n^{\al}(r).
$$
Finally, since $K_{\nu}=K_{-\nu}$ (see \cite[(5.7.10)]{Lebedev}), we obtain the desired result.
\end{proof}

We establish now the relation between $c_{n,\dl}^{1/2}(\si)$ and $c_{n,\dl}^{1/2}(-\si)$. Actually, the result below follows from \cite[Proposition 3.3]{RT}. Nevertheless, we will provide a different, self-contained proof.

\begin{lem}
\label{lem:connecting}
For $ 0 < \si < 1 $, $\dl>0$, and $\al>-1/2$, we have
$$
c_{n,\dl}^{1/2}(-\si)  =   \dl^{\si}\frac{\Gamma\big(\frac{\al+2+\si}{2}\big)^2}{\Gamma\big(\frac{\al+2-\si}{2}\big)^2} S_n^{\al,-\si},
$$
where $S_n^{\al,\si}$ was defined in \eqref{eq:Salpha}.
\end{lem}

\begin{proof}
From \eqref{eq:coefficient}, the identity to be proved is equivalent to
\begin{multline*}
L\Big( \frac{\dl}{2}, \frac{4n+2\al+2}{4}+\frac{1-\si}{2},\frac{4n+2\al+2}{4}+\frac{1+\si}{2}\Big)\\
= \dl^{\si}\frac{\Gamma\big(\frac{4n+2\al+2}{4}+\frac{1-\si}{2}\big)}
{\Gamma\big(\frac{4n+2\al+2}{4}+\frac{1+\si}{2}\big)}L\Big( \frac{\dl}{2}, \frac{4n+2\al+2}{4}+\frac{1+\si}{2},\frac{4n+2\al+2}{4}+\frac{1-\si}{2}\Big).
\end{multline*}
Let us call $w:=\frac{4n+2\al+2}{4}$, so that the identity above reads as
$$
L\Big(  \frac{\dl}{2}, w+\frac{1-\si}{2},w+\frac{1+\si}{2}\Big)
= \frac{\Gamma\big(w+\frac{1-\si}{2}\big)}
{\Gamma\big(w+\frac{1+\si}{2}\big)}L\Big( \frac{\dl}{2}, w+\frac{1+\si}{2},w+\frac{1-\si}{2}\Big).
$$
The confluent hypergeometric function of second type is given by
$$
U(a,b,x):=\frac{1}{\Gamma(a)}\int_0^{\infty}e^{-xt}t^{a-1}(1+t)^{b-a-1}\,dt.
$$
It is known that $U(a,b,x)$ is the solution of the differential equation $xy''+(b-x)y'-ay=0$. Moreover, it satisfies
\begin{equation}
\label{eq:Us}
U(a,b,x)=x^{1-b}U(a-b+1,2-b,x).
\end{equation}
Actually, both functions above are solutions to the differential equation and  asymptotically behave like $x^{-a}$ as $x\to \infty$ (for these properties see \cite{Olde}). Then, it is clear that, by \eqref{eq:Us},
$$
L(a,b,c)=e^{-a}\Gamma(b)U(b,b-c+1,2a)=e^{-a}\Gamma(b)(2a)^{c-b}U(c,c-b+1,2a)=\frac{\Gamma(b)}{\Gamma(c)}(2a)^{c-b}L(a,c,b).
$$
Then, by taking $a=\dl/2$, $b=w+\frac{1-\si}{2}$, $c=w+\frac{1+\si}{2}$, we obtain the desired result.

\end{proof}

We end up with the following results, that will motivate the definition of the ground state representation.
\begin{thm}
\label{prop:Lsweight}
Let $0<\si<1$, $\dl>0$, and $\al>-1/2$. Then, for any $f\in C_0^{\infty}(0,\infty)$ we have
$$
\int_0^{\infty}L_{\al,\si}f(r)w_{\al,-\si}^{\dl}(r)r^{2\al+1}\,dr=(4\dl)^{\si} \frac{\Gamma\big(\frac{\al+2+\si}{2}\big)^2}{\Gamma\big(\frac{\al+2-\si}{2}\big)^2} \int_0^{\infty}f(r)w_{\al,\si}^{\dl}(r)r^{2\al+1}\,dr.
$$
\end{thm}
\begin{proof}
By Proposition \ref{eq:LagseriesK} we have
\begin{equation}
\label{eq:weightspectral}
w_{\al,-\si}^{\dl}=\sum_{n=0}^{\infty}c_{n,\dl}^{1/2}(-\si)\varphi_n^{\al}.
\end{equation}
Then, in view of the spectral decomposition of $L_{\al,\si}$ \eqref{eq:spectralLsi}, the Laguerre decomposition of $w_{\al,-\si}$ in \eqref{eq:weightspectral}, \eqref{eq:convoLaguerres}, and Lemma \ref{lem:connecting} we have
$$
L_{\al,\si}w_{\al,-\si}^{\dl}=(4\dl)^{\si} \frac{\Gamma\big(\frac{\al+2+\si}{2}\big)^2}{\Gamma\big(\frac{\al+2-\si}{2}\big)^2} \sum_{n=0}^{\infty}c_{n,\dl}^{1/2}(\si)\varphi_n^{\al}=(4\dl)^{\si} \frac{\Gamma\big(\frac{\al+2+\si}{2}\big)^2}{\Gamma\big(\frac{\al+2-\si}{2}\big)^2} w_{\al,\si}^{\dl}.
$$
With this, we  easily conclude the proof of the theorem.
\end{proof}
Let $G_{\al,\si}$ be the function defined by
\begin{equation}
\label{eq:Gfundamental}
G_{\al,\si}:= \frac{2}{\Gamma(\al+1)}\sum_{n=0}^{\infty}S_n^{\alpha,\si}\varphi_n^{\al},
\end{equation}
where $S_n^{\alpha,\si}$ is the constant defined in \eqref{eq:Salpha}.
A direct computation, using \eqref{eq:convoLaguerres} and Lemma \ref{lem:connecting}, yields the following.
\begin{prop}
\label{prop:CH}
Let $0<\si<1$. Then,
$$
w_{\al,-\si}^{\dl}\ast_{\al}G_{\al,-\si}=4^{\si} \frac{\Gamma\big(\frac{\al+2+\si}{2}\big)^2}{\Gamma\big(\frac{\al+2-\si}{2}\big)^2} w_{\al,\si}^{\dl}.
$$
\end{prop}

We finish this subsection by showing that the function $G_{\al,\si}$ given in \eqref{eq:Gfundamental} is indeed a fundamental solution for the operator $L_{\al,\si}$, and we obtain a explicit expression for it.

\begin{thm}
\label{thm:fundamental}

The function $H_{\al,\si}=\frac{\sqrt{2}\Gamma(\al+1/2)}{4^\si \Gamma(\alpha+1)}G_{\al,\si}$ verifies that $L_{\al,\si}H_{\al,\si}=\delta_0$ where $\delta_0$ is the Dirac delta distribution with support at $0$. Moreover,
$$
G_{\al,\si}(r)=\frac{2^{\si+\al}\Gamma\big(\frac{\al-\si}{2}\big)}{\sqrt{\pi}\Gamma(\si)}
r^{-(\al+1-\si)}K_{(\al+1-\si)/2}(r^2/2).
$$
\end{thm}
\begin{proof}
Since
\[
\mathcal{T}_0^\alpha f(s)=\frac{\Gamma(\alpha+1/2)}{\sqrt{2}\Gamma(\alpha+1)}f(s),
\]
we have
\[
f\ast_\al \varphi_j^\alpha(0)=\frac{\Gamma(\alpha+1/2)}{\sqrt{2}\Gamma(\alpha+1)}\int_0^\infty f(s)\varphi_j^\alpha(s)s^{2\alpha+1}\, dx
\]
and
\[
\langle \varphi_k^\alpha,f\ast_\al \varphi_j^\alpha \rangle_{L^2(\mathbb{R}^+,d\mu_\alpha)}=\frac{(\Gamma(\alpha+1))^2}{\sqrt{2}\Gamma(\alpha+1/2)}\delta_{jk}f\ast_\al \varphi_j^\alpha(0).
\]
Then the first statement of the theorem follows from the spectral decomposition of $L_{\al,\si}$ \eqref{eq:spectralLsi} together with the definition of $G_{\al,\si}$ in \eqref{eq:Gfundamental}. Indeed,
\[
\langle L_{\al,\si}H_{\al,\si}, f  \rangle_{L^2(\mathbb{R}^+,d\mu_\alpha)}=
\langle H_{\al,\si}, L_{\al,\si}f  \rangle_{L^2(\mathbb{R}^+,d\mu_\alpha)}=\frac{2}{\Gamma(\alpha+1)}\sum_{j=0}^\infty f\ast_\al \varphi_j^\alpha(0)=f(0).
\]

For the second one, we make use of the formula \eqref{eq:GR0} with $\mu=(4n+2\al+2)$, $\nu=\si-1$ and $\beta=2$, so we have
$$
G_{\al,\si}(r)=\frac{2^\si}{\Gamma(\si)}\int_0^{\infty}q_{t,\al}(r)(\sinh 2t)^{\si-1}\,dt,
$$
where $q_{t,\al}$ is the heat kernel given in \eqref{eq:heatkernelLag}.
Thus
$$
G_{\al,\si}(r)=\frac{4^\si}{\Gamma(\si)}\int_0^{\infty}(\sinh 2t)^{-\al+\si-2}e^{-(\coth 2t)r^2/2}\,dt=:\frac{4^\si}{\Gamma(\si)}I(r).
$$
Let us compute the integral above. The change of variable $\coth 2t=z+1$ yields
$$
I(r)=2^{-1}e^{-r^2/2}\int_0^{\infty}e^{-zr^2/2}z^{(\al-\si)/2}(z+2)^{(\al-\si)/2}\,dz.
$$
Now we use the identity (see \cite[2.3.6.10, p. 324]{Prudnikov1})
$$
\int_0^{\infty}e^{-px}x^{\mu-1}(x+z)^{\mu-1}\,dx=\frac{\Gamma(\mu)}{\sqrt{\pi}}(p/z)^{1/2-\mu}e^{pz/2}K_{\mu-1/2}(pz/2),
$$
valid for $\Re\mu, \Re p>0; |\arg z|<\pi$, with $\mu=(\al-\si)/2+1$, $p=r^2/2$ and $z=2$, so that
$$
I(r)=2^{-1}\frac{\Gamma((\al-\si)/2)2^{\al+1-\si}}{\sqrt{\pi}}r^{-(\al+1-\si)}
K_{(\al+1-\si)/2}(r^2/2)
$$
and we get the conclusion.
\end{proof}

\subsection{The ground state representation for $L_{\al,\si}$ and proof of Theorem \ref{thm:HardyLaguerre}}
\label{subsec:CH2}

We proceed with the next step, namely, to get a ground state representation for the operator $L_{\al,\si}$. Let us set
$$
H_{\al,\si}^{\dl}[f]:=\langle L_{\al,\si}f,f\rangle-A_{\al,\si}^{\dl}
\int_{0}^{\infty}|f(r)|^2\frac{w_{\al,\si}^{\dl}(r)}{w_{\al,-\si}^{\dl}(r)}\,d\mu_{\al}(r),
$$
where
\begin{equation}
\label{eq:Aconstant}
A_{\al,\si}^{\dl}:=(4\dl)^{\si} \frac{\Gamma\big(\frac{\al+2+\si}{2}\big)^2}{\Gamma\big(\frac{\al+2-\si}{2}\big)^2}
\end{equation}
(observe that $A_{\al,\si}^{\dl}=\frac{4^{\si}}{\dl^{\si}}(B^{\dl}_{\al,\si})^2$, with $B_{\al,\si}^{\dl}$ defined in \eqref{eq:Bconstant}).

\begin{thm}
\label{thm:groundstateLag}
Let $0<\si<1$. If $f\in C_0^{\infty}(0,\infty)$ and $g(r)=f(r)(w_{\al,-\si}^{\dl}(r))^{-1}$ then
$$
H_{\al,\si}^{\dl}[f]
=\frac{2^{\si-1}}{|\Gamma(-\si)|}\int_{0}^{\infty}\int_{0}^{\infty}|g(r)-g(s)|^2
\mathcal{K}_{\al,\si}(r,s)w_{\al,-\si}^{\dl}(r)w_{\al,-\si}^{\dl}(s)
\,d\mu_{\al}(r)\,d\mu_{\al}(s).
$$
\end{thm}
\begin{proof}
We proceed as in Theorem \ref{thm:groundstate}, by polarizing the integral representation in Proposition \ref{prop:integRep} and taking $G(x)=w_{\al,-\si}^{\dl}(r)$ and $F(x)=\frac{|f(x)|^2}{w_{\al,-\si}^{\dl}(r)}$. Moreover, Theorem \ref{prop:Lsweight} and \eqref{eq:weight} allow us to write
$$
\langle L_{\al,\si}F, G\rangle =(4\dl)^{\si} \frac{\Gamma\big(\frac{\al+2+\si}{2}\big)^2}{\Gamma\big(\frac{\al+2-\si}{2}\big)^2} \int_{0}^{\infty}|f(r)|^2\frac{w_{\al,\si}^{\dl}(r)}{w_{\al,-\si}^{\dl}(r)}\,d\mu_{\al}(r).
$$
The conclusion follows as in Theorem \ref{thm:groundstate}.
\end{proof}

Observe that if we take $f=w_{\al,-\si}^{\dl}(r)$ in Theorem \ref{thm:HardyLaguerre}, both sides of the inequality reduce to
$$
(4\dl)^{\si} \frac{\Gamma\big(\frac{\al+2+\si}{2}\big)^2}{\Gamma\big(\frac{2\al+2-\si}{2}\big)^2} \int_{0}^{\infty}w_{\al,-\si}^{\dl}(r)w_{\al,\si}^{\dl}(r)\,d\mu_{\al}(r),
$$
so in that case the constant $(4\dl)^{\si} \frac{\Gamma\big(\frac{\al+2+\si}{2}\big)^2}{\Gamma\big(\frac{\al+3-\si}{2}\big)^2}$
is optimal in our inequality.

\begin{proof}[Proof of Theorem \ref{thm:HardyLaguerre}]
From the ground state representation $H_{\al,\si}^{\delta}[f]$, it is clear that
$$
\langle L_{\al,\si}f,f\rangle\ge (4\dl)^{\si}\frac{\Gamma\big(\frac{\al+2+\si}{2}\big)^2}{\Gamma\big(\frac{\al+2-\si}{2}\big)^2} \int_{0}^{\infty}|f(r)|^2\frac{w_{\al,\si}^{\dl}(r)}{w_{\al,-\si}^{\dl}(r)}\,d\mu_{\al}(r).
$$
It is known that $K_{\nu}(x)$ is an increasing function of $\nu$ for $x>0$ (see \cite[p. 226]{NU}). With this, and recalling the definition of $w_{\al,\si}^{\dl}$ in \eqref{eq:weight}, we have
$$
\langle L_{\al,\si}f,f\rangle\ge (4\dl)^{\si}\frac{\Gamma\big(\frac{\al+2+\si}{2}\big)^2}{\Gamma\big(\frac{\al+2-\si}{2}\big)^2} \int_{0}^{\infty}|f(r)|^2\frac{w_{\al,\si}^{\dl}(r)}{w_{\al,-\si}^{\dl}(r)}\,d\mu_{\al}(r)\ge \dl^{\si}\frac{\Gamma\big(\frac{\al+2+\si}{2}\big)}{\Gamma\big(\frac{\al+2-\si}{2}\big)} \int_{0}^{\infty}\frac{|f(r)|^2}{(\dl+r^2)^{\si}}\,d\mu_{\al}(r),
$$
which is the required inequality.
\end{proof}
\section{A Hardy inequality for the fractional Dunkl--Hermite operator}
\label{sec:hardyDH}

\subsection{The general Dunkl setting}
\label{subsec:generalDunkl}

We will introduce some basic facts concerning the general Dunkl setting. A complete picture of the Dunkl's theory can be found in \cite{DU, DUX, RO}. We also refer the reader to the survey article \cite{RO3}.

Let us use
$\La \cdot , \cdot \Ra $ for the standard
 inner product on $\R^d $.
 For $\nu\in \R^d\setminus\{0\}$, we denote by $\si_{\nu}$ the orthogonal reflection in the hyperplane perpendicular to $\nu$, namely,
 $$
\si_{\nu}(x)= x - 2 \; \frac{\La \nu, x \Ra}{|\nu |^2} \nu .
$$
We say that a finite subset $R\subset  \R^d\setminus\{0\}$ is a reduced root system if, for all $\nu\in R$, then $\si_{\nu}(R)=R$ and $\R\nu\cap R=\{\pm\nu\}$. Each root system can be written as a disjoint union $R=R_+\cup (-R_+)$,
where $R_+$ and $-R_+$ are separated by a hyperplane through the origin. Such $R_+$ is called the set of all positive roots in $R$.
 The group $G$ generated by the reflections
 $\{\si_{\nu}:\nu\in R \}$ is called the reflection group or Coxeter group
 associated with $R$. A function
 \begin{equation}
 \label{eq:k}
 \K:R\to [0,\infty)
 \end{equation}
 which is invariant under the action of $G$ on the root system $R$ is called a
 multiplicity function.
  Let $T_j$, $j=1,2, \ldots d$, be the difference
 -differential operators  defined by
 $$ T_jf(x) = \frac{\pa f}{\pa x_j}(x) +\sum_{\nu \in R_+}\K(\nu)\nu _j
 \frac{f(x)-f(\si_{\nu}x)}{\La \nu, x \Ra}.
 $$
  These operators, known as Dunkl
 operators, form a family of commuting operators. The Dunkl
 Laplacian $\Dk $ is then defined to be the operator
 \begin{equation}
 \label{eq:DunklLap}
  \Dk = \sum _{j=1}^d T_j^2
  \end{equation}
 which can be explicitly calculated, see
 \cite[Theorem 4.4.9]{DUX}. It is known that the operators $T_j$ have a joint eigenfunction
$E_{\K}(x,y)$ satisfying
 $$
 T_j E_{\K}(x,y)=y_j E_{\K}(x,y), \qquad  j=1,\ldots,d.
 $$
 The function $(x,y)\mapsto E_{\K}(x,y)$ is called the \textit{Dunkl kernel} or the \textit{generalized exponential
 kernel} on $\R^d\times \R^d$, which is the generalization of the exponential function $e^{\langle x,y\rangle}$.
 Associated with the root system $R$ and the multiplicity function $\K $, the weight function $h_{\K}(x)$
 is defined by
 $$
 h_{\K}(x):=\prod_{\nu \in R_+}|\La x, \nu \Ra |^{2\K(\nu)}.
 $$
The nonnegative real number
\begin{equation}
\label{gamma}
\gm = \sum_{\nu \in R_+}\K(\nu)
\end{equation}
defined in terms of the multiplicity
 function $\K(\nu)$ plays an important role in Dunkl theory. Note that $h_{\K}(x)$ is
 homogeneous of degree $2\gm $.
 For a radial function $f$ in $L^1(\R^d, h_{\K})$, there exists a function $F$ on $[0,\infty)$ such that $f(x)=F(|x|)$, for all $x\in \R^d$.
In view of  the homogeneity of $h_{\K}(x)$, it follows that the function $F$ is integrable with respect to the measure $r^{d+2\gm-1}\, dr$ on $[0,\infty)$ and we have
 \begin{equation*}
 \begin{split}
 \int_{\R^d}f(x)h_{\K}(x)\,dx&=\int_0^{\infty}\Big(\int_{\sd}f(rx') h_{\K}(rx')\,d\sigma(x')\Big)r^{d-1}\,dr\\
& =\int_0^{\infty}\Big(\int_{\sd}h_{\K}(rx')\,d\sigma(x')\Big)F( r)r^{d-1}\,dr\\
&=d_{\K}\int_0^{\infty}F(r)r^{d+2\gm-1}\,dr,
 \end{split}
 \end{equation*}
 where
 $$
d_{\K}:=\int_{\sd}h_{\K}(x')\,d\sigma(x')=\frac{c_{\K}^{-1}}{2^{\frac{d}{2}+\gm-1}\Gamma\big(\frac{n}{2}+\gm\big)},
 $$
and $c_{\K}$ is the Mehta-type constant
 $$
 c_{\K}:=\Big(\int_{\R^d}e^{-\frac{|x|^2}{2}}h_{\K}(x)\,dx\Big).
 $$

\subsection{$h$-harmonic expansions}
\label{subsec:hharmonic}

The theory of spherical harmonics which deals with expansions of functions in $L^2(\sd,d\si) $ in terms of spherical harmonics has an analogue for the space $ L^2(\sd, h_\K  d\si)$.  Given a function $f \in L^2(\sd,h_\K d\si) $ we can expand it in terms of the so called \textit{spherical $h$-harmonics} (or just \textit{$h$-harmonics}). We refer the reader to \cite[Chapter 5]{DUX} concerning $h$-harmonics. These are the restrictions of solid $h$-harmonics to $\sd$ where by solid $h$-harmonics we mean
homogeneous polynomials $P(x) $
 satisfying $\Dk P(x)=0 $.  The $h$-harmonics are analogues of spherical harmonics and defined using $\Dk $
 in place of $\D$. Let $\Hc_{m}^d $ be the space of all $h$-harmonics of degree $m
 $. Then the space $L^2\big(\sn,h_{\K}(x')\, d\sigma(x')\big) $ is the orthogonal
 direct sum of the finite dimensional spaces $\Hc_{m}^d $ over $m=0,1,2,\ldots$. Thus there is an orthonormal basis
 $\{Y_{m,j}^h:  j= 1,2,\ldots,\,d(m), m = 0,1,2,\ldots  \}$,
 where
 \begin{equation*}
 d(m)=\operatorname{dim} (\Hc_{m}^d),
 \end{equation*}
 for $L^2\big(\sd,h_{\K}(x')\, d\sigma(x')\big)$
 so that for each $ m, \{Y_{m,j}^h: j=1,2,\ldots,d(m)
 \}$ is an orthonormal basis of $h$-harmonics  for $\Hc_{m}^d $. For $x\in \R^d$, take $x=rx'$, $r\in \R^+$ and $x'\in \sd$. The $h$-harmonic expansion of a function $f$ on $ \R^d $ is given by
 \begin{equation}
 \label{eq:hexpansion}
   f(rx')=\sum_{m=0}^\infty\sum_{j=1}^{d(m)}f_{m,j}(r) Y_{m,j}^h(x')=:\sum_{m,j}f_{m,j}(r)Y_{m,j}^h(x'),
\end{equation}
 where the $h$-harmonic coefficients are
 \begin{equation*}
    f_{m,j}(r)=\int_{\sd}f(rx')Y_{m,j}^h(y')h_{\K}(y')d\si(y').
\end{equation*}

In the Dunkl setting we have a \textit{Funk--Hecke formula for $h$-harmonics}. The classical Funk--Hecke formula for spherical harmonics states
 the following. For any continuous function $f $ on $[-1, 1] $ and a (standard)
 spherical harmonic $Y_{m,j}$ of degree $m $ we have
 $$
 \int_{\sd} f(\La x', y'\Ra )Y_{m,j}(y')d\si (y')= \Lambda_m (f)Y_{m,j}(x')
 $$
 where $\Lambda_m (f) $ is a constant defined by
 $$
 \Lambda_m (f)= \frac{\om_d\Gm(d/2)}{\sqrt{\pi}\Gm((d-1)/2)}\int_{-1}^1
 f(u)P_m^{\frac{d}{2}-1}(u)(1-u^2)^{\frac{d-3}{2}}du .
 $$
 Here $P_m^{\lambda} $ stands for the normalized ultraspherical polynomials of type
 $\lambda>-\tfrac12$ and degree $m$, and  $\om_d:=\int_{\sd}d\si(\om)$. To state the Funk--Hecke formula for $h$-harmonics, we need to recall the
 intertwining operator in the Dunkl setting. It is known that there is an
 operator $V_{\K}$ satisfying $T_jV_{\K}=V_{\K} \frac{\pa}{\pa x_j} $. However, the explicit
 form of $V_{\K}$ is not known, except in a couple of simple cases. In particular, the Dunkl kernel
 is given by $E_{\K}(x,y) = V_{\K}e^{\La \cdot ,y \Ra}(x) $. For $d\ge1$ and $\gm$ as in \eqref{gamma}, let us recall the definition of the constant $\la$ as
\begin{equation}
\label{eq:lambda}
\lambda:=\frac{d}{2}+\gm-1.
\end{equation}
 The Funk--Hecke formula
 for $h$-harmonics is as follows (see \cite[Theorem 7.2.7]{DX} or \cite[Theorem 5.3.4]{DUX}).
 \begin{thm}
 \label{Th:FunkHeckeh}
  Let $f$ be a continuous function defined on $[-1, 1]$ and let
  $\la$ be as in \eqref{eq:lambda}. Then for every $Y_{m,j}^h \in \Hc_{m}^d$,
  $$
  \int_{\sd}V_{\K}f(\La x',\cdot \Ra)(y')Y_{m,j}^h(y')h_{\K}(y')d\si(y')=\Lambda_m^{\K}(f)Y_{m,j}^h(x')
  $$
  where $\Lambda_m^{\K}(f) $ is  defined by
  $$
  \Lambda_m^{\K}(f)= \frac{\om_d^{\K}\Gm(\la+1)}{\sqrt{\pi}\Gm(\la+1/2)}
  \int_{-1}^1f(u)P_m^{\la}(u)(1-u^2)^{\la-\frac{1}{2}}\,du
  $$
  with
  \begin{equation*}
  \om_d^{\K}:=\int_{\sd}h_{\K}(\om)d\si(\om).
  \end{equation*}
 \end{thm}
By applying Theorem \ref{Th:FunkHeckeh} to  the function $f(t) =e^{rst}$, $r, s \geq 0$,
 and using  the fact $V_{\K}f(\La x', y'\Ra) = E_{\K}(rx',
  sy')$, we immediately obtain the following.
  \begin{cor}[Funk--Hecke for Dunkl kernel]
  \label{eq:Hecke}
  Let $\la$ be as in \eqref{eq:lambda}. Then for every $Y_{m,j}^h \in \Hc_{m}^d$,
 \begin{multline*}
 \int_{\sd}E_{\K}(rx',
 sy')Y_{m,j}^h(y')h_{\K}(y')\,d\si(y') \\
 =\frac{\om_d^{\K}\Gm\big (\frac{d}{2}+\gm\big )}{\sqrt{\pi}\Gm\big (\frac{d-1}{2}+\gm\big )}
 \bigg(\int_{-1}^{1}e^{rsu}P_m^{\la}(u)(1-u^2)^{\la-\frac12}\,du\bigg)
 Y_{m,j}^h(x').
 \end{multline*}

  \end{cor}
The following identity was proved in \cite[Lemma 7.2]{BRT}.
\begin{lem}[\cite{BRT}, Lemma 7.2]
\label{eq:UltraBessel}
Let $z\in \C$ and $\la>-\frac{1}{2}$. Then the following holds
\begin{equation*}
\int_{-1}^{1}e^{zu}P_m^{\la}(u)(1-u^2)^{\la-1/2}\,du =\sqrt{\pi}\Gm (\la+1/2 )(z/2 )^{-\la}I_{\la+m}(z), \quad  m=0,1,\ldots.
\end{equation*}
\end{lem}

\subsection{The Dunkl--Hermite operator and the heat semigroup}

Let us introduce the framework of the Dunkl--Hermite operator (which is also known as the Dunkl-harmonic oscillator)
 \begin{equation}
 \label{DHO}
 H_{\K}:=-\D_{\K}+|x|^2,
 \end{equation}
 where $\D_{\K}$ stands for the Dunkl--Laplacian in $\R^d$ \eqref{eq:DunklLap}.  The parameter $\K$ is the multiplicity function defined in \eqref{eq:k}. When $\K\equiv 0$, the operator $H_{\K}$ becomes the classical harmonic oscillator $-\Delta+|x|^2$. The study of $H_{\K}$ was initiated by M. R\"osler \cite{RO, RO2}.

For each $\mu\in \Na^d$, we consider the \textit{generalized Hermite functions} (or \textit{Dunkl--Hermite functions}) $\Phi_{\mu,\K}$. The system $\{\Phi_{\mu,\K}:\mu\in \Na^d\}$ is orthonormal and complete in $L^2(\R^d,h_{\K}(x))$, cf. \cite[Corollary 3.5 (ii)]{RO}. Furthermore, $\Phi_{\mu,\K}$ are eigenfuntions of the operator \eqref{DHO} with eigenvalues $(2|\mu|+d+2\gamma)$, that is
\begin{equation}
\label{eq:eigenvalues}
H_{\K}\Phi_{\mu,\K}=(2|\mu|+d+2\gamma)\Phi_{\mu,\K},
\end{equation}
where $\gamma$ is as in \eqref{gamma}, and $|\mu|=\sum_{\ell=1}^d\mu_{\ell}$. For $\K\equiv0$, $\Phi_{\mu,0}$ become the usual Hermite functions, see \cite[p. 521]{RO}. Precise definitions and detailed description on results concerning generalized Dunkl--Hermite functions can be found in \cite{RO}.

For a function $f \in L^2(\R^d, h_{\K})$ we have the orthogonal expansion,
\begin{equation*}
f=\sum_{n=0}^{\infty}P_{n,\K}f,
\end{equation*}
which converges in $L^2(\R^d, h_{\K})$.
Here the spectral projections are given by
\begin{equation*}
P_{n,\K}f= \sum_{|\mu|= n}\langle f, \Phi_{\mu,\K}\rangle_{\K} \Phi_{\mu,\K}
\end{equation*}
with $\langle\cdot, \cdot\rangle_{\K}$ standing for the inner product in $L^2(\R^d, h_{\K}\,dx)$.  More precisely, $P_{n,\K}$ is the orthogonal projection associated to the eigenspace corresponding to the eigenvalue $(2n+d+2\gamma)$ of $H_{\K}$.

We state and prove the following theorem, that relates the projections $P_{n,\K}$ with the Laguerre convolution in Section \ref{sec:HardyLaguerre}. Such kind of an identity is called  \textit{Hecke--Bochner} identity.

\begin{thm}[Hecke--Bochner for Dunkl--Hermite projections]
\label{HB-DHProjections}
Let $f(x)=f_0(|x|)Y_{m}^h(x')$ where $Y_m^{h}(x')$ is a $h-$harmonic of degree $m$. Then one has $P_{2n+m,\K}f(x)=
F_n(|x|)Y_{m}^h(x')$ where
$$
F_n(r)=\frac{2}{\Gamma(\al+1)}r^m\big[\big((\cdot)^{-m}f_0(\cdot)
\big)\ast_{\la+m}\varphi_n^{\la+m}\big](r).
$$
with $\la$ as in \eqref{eq:lambda}. For other values of $\ell$, $P_{\ell,\K}f = 0$.
\end{thm}
\begin{proof}
The proof is similar to the proof of \cite[Theorem 3.4.1]{ST}.
Here we need to use the \textit{Mehler's formula for the Dunkl--Hermite functions} (see \cite[Theorem 3.12]{RO}): for $|w|<1$, one has
 \begin{equation}
 \label{eq:MehlerGen1}
 \sum_{\mu \in \mathbb{N}^d}\Phi_{\mu,\K}(x)\Phi_{\mu,\K}(y)w^{|\mu |}
 =\frac{2}{\om_d^{\K}\,\Gm(\la+1 )}(1-w^2)^{-(\la+1)}\exp\Big\{-\frac{1}{2}\Big (\frac{1+w^2}{1-w^2}\Big )
 (|x|^2+|y|^2)\Big\}E_{\K}\Big (\frac{2wx}{1-w^2}, y \Big).
 \end{equation}
We provide some details of the proof. Taking into account \eqref{eq:MehlerGen1}, if follows that
\begin{align*}
\sum_{n=0}^{\infty}&P_{n,\K}f(x)w^n=\frac{2}{\om_d^{\K}\,\Gm(\la+1 )}(1-w^2)^{-(\la+1)}\\
 &\qquad \qquad \times\int_{\R^d}\exp\Big\{-\frac{1}{2}\Big (\frac{1+w^2}{1-w^2}\Big )
 (|x|^2+|y|^2)\Big\}
 E_{\K}\Big (\frac{2wx}{1-w^2}, y \Big)f(y)h_{\K}(y)\,dy\\
 &=\frac{2}{\om_d^{\K}\,\Gm(\la+1 )}(1-w^2)^{-(\la+1)}\exp\Big\{-\frac{1}{2}\Big (\frac{1+w^2}{1-w^2}\Big )
 r^2\Big\}\\
 &\qquad \times\int_{0}^{\infty}\Big(\int_{\sd}E_{\K}\Big (\frac{2wrx'}{1-w^2}, sy' \Big)Y_m^h(y')h_{\K}(y')\,d\si(y')\Big)\exp\Big\{-\frac{1}{2}\Big (\frac{1+w^2}{1-w^2}\Big )
 s^2\Big\}f_0(s)s^{2\la+1}\,ds\\
 &=\frac{2}{\sqrt{\pi}\Gm (\la+1/2 )}(1-w^2)^{-(\la+1)}\exp\Big\{-\frac{1}{2}\Big (\frac{1+w^2}{1-w^2}\Big )
 r^2\Big\}Y_m^h(x')\\
 &\qquad \times\int_{0}^{\infty}\Big(\int_{-1}^1\exp\Big\{\frac{2wrs}{1-w^2}u\Big\}P_m^{\la}(u)(1-u^2)^{\la-\frac12}\,du\Big)\exp\Big\{-\frac{1}{2}\Big (\frac{1+w^2}{1-w^2}\Big )
 s^2\Big\}f_0(s)s^{2\la+1}\,ds
 \end{align*}
 where the last equality is true in view of the Funk--Hecke formula in Corollary \ref{eq:Hecke}. In its turn, by Lemma \ref{eq:UltraBessel}, this last expression equals
\begin{align*}
2&(1-w^2)^{-1}(wr)^{-\la}\exp\Big\{-\frac{1}{2}\Big (\frac{1+w^2}{1-w^2}\Big )
 r^2\Big\}Y_m^h(x')\\
 &\qquad \times\int_0^{\infty}\exp\Big\{-\frac{1}{2}\Big (\frac{1+w^2}{1-w^2}\Big )
 s^2\Big\}f_0(s)I_{\la+m}\Big(\frac{2wrs}{1-w^2}\Big)s^{\la+1}\,ds\\
 &=2(wr)^{m}Y_m^h(x')\\
 &\qquad \times\int_0^{\infty}(1-w^2)^{-1}(rsw)^{-\la-m}\exp\Big\{-\frac{1}{2}\Big (\frac{1+w^2}{1-w^2}\Big )
 (r^2+s^2)\Big\}I_{\la+m}\Big(\frac{2wrs}{1-w^2}\Big)f_0(s)s^{2\la+1+m}\,ds.
\end{align*}
Now, using the generating function identity for $\varphi_n^{\delta}(r)$ given in \eqref{eq:generating},
we have that (by using also the identity \cite[(6.1.28)]{ST} for the Laguerre translation in the second equality below)
\begin{align*}
\sum_{n=0}^{\infty}P_{n,\K}f(x)w^n&=2\sum_{n=0}^{\infty}w^{2n+m}r^m
\int_0^{\infty}\frac{\Gamma(n+1)}{\Gamma(n+\la+m+1)}
\varphi_n^{\la+m}(r)\varphi_n^{\la+m}(s)f_0(s)s^{2\la+1+m}\,ds~~Y_m^h(x')\\
&=\frac{2}{\Gamma(\al+1)}\sum_{n=0}^{\infty}w^{2n+m}r^m\int_0^{\infty}
\mathcal{T}_r^{\la+m}\varphi_n^{\la+m}(s)f_0(s)s^{2\la+1+m}\,ds~~Y_m^h(x')\\
&=\frac{2}{\Gamma(\al+1)}\sum_{n=0}^{\infty}w^{2n+m}r^m\int_0^{\infty}
\mathcal{T}_r^{\la+m}\varphi_n^{\la+m}(s)s^{-m}f_0(s)s^{2\la+2m+1}\,ds~~Y_m^h(x')\\
&=\frac{2}{\Gamma(\al+1)}\sum_{n=0}^{\infty}w^{2n+m}r^m\big[\big((\cdot)^{-m}f_0(\cdot)
\big)\ast_{\la+m}\varphi_n^{\la+m}\big](r)~~Y_m^h(x').
\end{align*}
The conclusion follows by comparing coefficients on both sides of the equality.
\end{proof}

The solution to the heat equation associated to the Dunkl--Hermite operator, i.e.
$$
\frac{\partial}{\partial t}u(x,t)=-H_{\K}u(x,t), \quad u(x,0)=f(x), \quad t>0, \, x\in \R^d,
$$
is given by $u(x,t)=:e^{-tH_{\K}}f(x)=:T_t^{\K}f(x)$, where $ T_t^{\K},~ t\ge 0$ is the Dunkl-Hermite semigroup  generated by $H_{\K},$ see \cite{A, RO} . In terms of the spectral decomposition,
$$
T_t^{\K}f=\sum_{n=0}^{\infty}e^{-t(2n+d+2\gamma)}P_{n,\K}f, \quad f\in L^2(\R^d,h_{\K}\,dx).
$$
The following result relates the  Dunkl-Hermite semigroup and the Laguerre  heat semigroup.

\begin{thm}
\label{thm:relatingheats}
Let $\la$ be as in \eqref{eq:lambda}. For $x\in \R^d$, $x=rx'$, with  $r\in \R^+$ and $x'\in \sd$, then
\begin{equation*}
T_t^{\K}f(x)=\sum_{m=0}^\infty\sum_{j=1}^{d(m)}r^mT_{\la+m,t}((\cdot)^{-m}f_{m,j})(r) Y_{m,j}^h(x').
\end{equation*}
\end{thm}
\begin{proof}
We have just to apply Theorem \ref{HB-DHProjections} to the spectral definition of $T_t^{\K}f(x)$. Indeed,
\begin{align*}
T_t^{\K}f(x)&=\sum_{n=0}^{\infty}e^{-t(2n+d+2\gamma)}P_{n,\K}f(x)\\
&=\sum_{n=0}^{\infty}e^{-t(2n+d+2\gamma)}P_{n,\K}\Big(\sum_{m,j}f_{m,j}(r)Y_{m,j}^h(x')\Big)\\
&=\frac{2}{\Gamma(\al+1)}\sum_{m,j}Y_{m,j}^h(x')\sum_{\ell=0}^{\infty}e^{-t(4\ell+2m+d+2\gamma)}r^m\big[\big((\cdot)^{-m}f(\cdot)
\big)\ast_{\la+m}\varphi_{\ell}^{\la+m}\big](r)\\
&=\sum_{m,j}r^mT_{\la+m,t}((\cdot)^{-m}f_{m,j})(r) Y_{m,j}^h(x').
\end{align*}
We have proved the theorem.
\end{proof}

\subsection{The Hardy inequality for $\mathbf{H}_{\K,\si}$: Proof of Theorem \ref{thm:HardyDunkl-Hermite}}
\label{subse:HardyHK}

Given $0<\si<1$, we define $\mathbf{H}_\si$ to be the operator
$$
\mathbf{H}_{\K,\si}:=\frac{\Gamma\big(\frac{H_{\K}}{2}+\frac{1+\si}{2}\big)}
{\Gamma\big(\frac{H_{\K}}{2}+\frac{1-\si}{2}\big)},
$$
so, in view of \eqref{eq:eigenvalues}, $\mathbf{H}_{\K,\si}$ corresponds to the spectral multiplier $
\frac{\Gamma\big(\frac{2|\mu|+d+2\gamma}{2}+\frac{1+\si}{2}\big)}
{\Gamma\big(\frac{2|\mu|+d+2\gamma}{2}+\frac{1-\si}{2}\big)}$. Then by Lemma~\ref{lem:numerical1} we have
\begin{align*}
\mathbf{H}_{\K,\si}f(x)&=\sum_{j=0}^{\infty}
\frac{\Gamma\big(\frac{2j+d+2\gamma}{2}+\frac{1+\si}{2}\big)}
{\Gamma\big(\frac{2j+d+2\gamma}{2}+\frac{1-\si}{2}\big)}P_{j,\K}f(x)\\
&=E_{\si}f(x)+\frac{2^{\si}}{|\Gamma(-\si)|}\int_{0}^{\infty}\big(f(x)-T_t^{\K}f(x)\big)
(\sinh t)^{-\si-1}\,dt.
\end{align*}
By Theorem \ref{thm:relatingheats} and the expansion of $f$ into $h$-harmonics \eqref{eq:hexpansion}, we can write
\begin{align*}
\mathbf{H}_{\K,\si} f(x)&=E_{\si}f(x)+\frac{2^{\si}}{|\Gamma(-\si)|}
\int_{0}^{\infty}\big(f(x)-T_t^{\K}f(x)\big)
(\sinh t)^{-\si-1}\,dt\\
&=\sum_{m,j}Y_{m,j}^h(x')\Big[E_{\si}f_{m,j}(r)\\
&\qquad +\frac{2^{\si}}{|\Gamma(-\si)|}
\int_{0}^{\infty}\big(f_{m,j}(r)-r^mT_{\la+m,t}((\cdot)^{-m}f_{m,j})(r)\big)
(\sinh t)^{-\si-1}\,dt\Big]\\
&=\sum_{m,j}Y_{m,j}^h(x')r^m\Big[E_{\si}r^{-m}f_{m,j}(r)\\
&\qquad  +\frac{2^{\si}}{|\Gamma(-\si)|}
\int_{0}^{\infty}\big(r^{-m}f_{m,j}(r)-T_{\la+m,t}((\cdot)^{-m}f_{m,j})(r)\big)
(\sinh t)^{-\si-1}\,dt\Big]\\
&=\sum_{m,j}Y_{m,j}^h(x')r^m\Big[E_{\si}g_{m,j}(r)
+\frac{2^{\si}}{|\Gamma(-\si)|}
\int_{0}^{\infty}\big(g_{m,j}(r)-T_{\la+m,t}g_{m,j}(r)\big)
(\sinh t)^{-\si-1}\,dt\Big]\\
&=\sum_{m,j}Y_{m,j}^h(x')r^m\Big[g_{m,j}(r)
\Big(E_{\si}+\frac{2^{\si}}{|\Gamma(-\si)|}F_{\la+m,\si}(r)\Big)\\
&\qquad+\frac{2^{\si}}{|\Gamma(-\si)|}
\int_{0}^{\infty}\big(g_{m,j}(r)-g_{m,j}(s)\big)\mathcal{K}_{\la+m,\si}(r,s)
\,d\mu_{\la+m}(s)\Big]\\
&=\sum_{m,j}Y_{m,j}^h(x')r^mL_{\la+m,\si}g_{m,j}(r),
\end{align*}
where we $g_{m,j}(r)=r^{-m}f_{m,j}(r)$ and $L_{\la+m,\si}$ is the modified fractional Laguerre operator defined spectrally in \eqref{eq:spectralLsi}.
With this, by Theorem \ref{thm:HardyLaguerre}, we have
\begin{align*}
\langle \mathbf{H}_{\K,\si} f,f\rangle_{L^2(\R^d,h_{\K}) }&=\sum_{m=0}^\infty\sum_{j=1}^{d(m)}\langle L_{\la+m,\si} g_{m,j},g_{m,j}\rangle_{L^2((0,\infty),d\,\mu_{\la+m}(r)) }\\
&\ge \sum_{m=0}^\infty\sum_{j=1}^{d(m)}A_{\la+m,\si}^{\dl}\int_{0}^{\infty}|g_{m,j}(r)|^2
\frac{w_{\la+m,\si}^{\dl}(r)}{w_{\la+m,-\si}^{\dl}(r)}\,d\mu_{\la+m}(r)\\
&=\sum_{m=0}^\infty\sum_{j=1}^{d(m)}A_{\la+m,\si}^{\dl}\int_{0}^{\infty}|f_{m,j}(r)|^2
\frac{w_{\la+m,\si}^{\dl}(r)}{w_{\la+m,-\si}^{\dl}(r)}\,d\mu_{\la}(r),
\end{align*}
where $A_{\la+m,\si}^{\dl}$ is the constant in \eqref{eq:Aconstant}, i.e., $A_{\la+m,\si}^{\dl}=(4\dl)^{\si} \frac{\Gamma\big(\frac{\la+m+2+\si}{2}\big)^2}
{\Gamma\big(\frac{\la+m+2-\si}{2}\big)^2}$.
Taking into account the definition of $w_{\al,\si}^{\dl}(r)$ in \eqref{eq:weight} we get
\begin{equation}
\label{eq:int1}
A_{\la+m,\si}^{\dl}\frac{w_{\la+m,\si}^{\dl}(r)}{w_{\la+m,-\si}^{\dl}(r)}
=\dl^{\si}\frac{\Gamma\big(\frac{\la+m+2+\si}{2}\big)}
{\Gamma\big(\frac{\la+m+2-\si}{2}\big)}\frac{K_{(\la+m+2+\si)/2}\big((\dl+r^2)/2\big)}
{K_{(\la+m+2-\si)/2}\big((\dl+r^2)/2\big)}(\dl+r^2)^{-\si}.
\end{equation}
Now, as observed by D. Yafaev in \cite{Yafaev}, for $0<x\le y$ and $m\ge 0$ we have that $\frac{\Gamma(m+y)}{\Gamma(m+x)}\ge\frac{\Gamma(y)}{\Gamma(x)}$,
so
\begin{equation}
\label{eq:int2}
\dl^{\si}\frac{\Gamma\big(\frac{\la+m+2+\si}{2}\big)}
{\Gamma\big(\frac{\la+m+2-\si}{2}\big)}\ge\dl^{\si}\frac{\Gamma\big(\frac{\la+2+\si}{2}\big)}
{\Gamma\big(\frac{\la+2-\si}{2}\big)}
=B_{\la,\si}^{\dl},
\end{equation}
where $B_{\la,\si}^{\dl}$ is the constant defined in \eqref{eq:Bconstant}.
Recall again that $K_{\nu}(x)$ is an increasing function of $\nu$ for $x>0$ (see \cite[p. 226]{NU}). Therefore, with this, and putting together \eqref{eq:int1} and \eqref{eq:int2} we have
\begin{align*}
\langle \mathbf{H}_{\K,\si} f,f\rangle_{L^2(\R^d,h_{\K}\,dx) }\ge&\sum_{m,j}A_{\la+m,\si}^{\dl}\int_{0}^{\infty}|f_{m,j}(r)|^2
\frac{w_{\la+m,\si}^{\dl}(r)}{w_{\la+m,-\si}^{\dl}(r)}\,d\mu_{\la}(r)\\
&\ge B_{\la,\si}^{\dl}
\sum_{m,j}\int_{0}^{\infty}|f_{m,j}(r)|^2
(\dl+r^2)^{-\si}\,d\mu_{\la}(r)\\
&= B_{\la,\si}^{\dl}
\sum_{m,j}\langle (\dl+|\cdot|^2)^{-\si/2} f_{m,j},(\dl+|\cdot|^2)^{-\si/2}f_{m,j}\rangle_{L^2((0,\infty),d\,\mu_{\la}(r)) }\\
&= B_{\la,\si}^{\dl}
\langle (\dl+|\cdot|^2)^{-\si/2} f,(\dl+|\cdot|^2)^{-\si/2}f\rangle_{L^2(\R^d,h_{\K})}\\
&= B_{\la,\si}^{\dl}\int_{\R^d}\frac{|f(x)|^2}{(\dl+|x|^2)^{\si}}\,h_{\K}(x)\,dx.
\end{align*}
This completes the proof of Theorem \ref{thm:HardyDunkl-Hermite}.



\subsubsection*{Acknowledgments.}
This work began in the summer of 2015 when the third author visited Universidad de La Rioja, Logro\~no. He wishes to thank \'Oscar and Luz for making the visit possible and also for the warm hospitality he enjoyed during his visit.

The authors thank the anonymous referee for the very careful reading of the paper and for valuable suggestions that greatly contributed to the final form of this work.


\end{document}